\documentclass[12pt,a4paper]{amsart}

\usepackage{amsmath,amssymb}
\usepackage{relsize}
\usepackage{enumerate}
\usepackage{url}
\usepackage{xcolor}
\usepackage{xspace}
\usepackage[colorlinks=true]{hyperref}
\usepackage{xpicture}

\usepackage{pgfplots}
\pgfplotsset{compat=1.15}
\usepackage{mathrsfs}
\usepackage{cancel}
\usetikzlibrary{arrows}
\usepackage[normalem]{ulem}

\usepackage{cancel}

\newtheorem{theorem}{Theorem}
\newtheorem{lemma}[theorem]{Lemma}
\newtheorem{proposition}[theorem]{Proposition}

\theoremstyle{definition}
\newtheorem{definition}[theorem]{Definition}
\newtheorem{example}[theorem]{Example}

\newtheorem{remark}[theorem]{Remark}

\newcommand{\R}{\mathbb{R}}

\newcommand{\PP}{\mathbb{P}}

\def\conf{\mathcal{C}}

\title[Nef divisors at infinity]{Nef divisors of surfaces given by pencils at infinity}
\author[C. Galindo, F. Monserrat, C.-J. Moreno-\'Avila and E. P\'erez-Callejo]{Carlos Galindo, Francisco Monserrat, Carlos Jes\'us Moreno-\'Avila and Elvira P\'erez-Callejo}
\curraddr{\texttt{Carlos Galindo:} Instituto
Universitario de Matem\'aticas y Aplicaciones de Castell\'on and
Departamento de Matem\'aticas, Universitat Jaume I, Campus de Riu
Sec. 12071 Castell\'{o} (Spain)\\
\texttt{Francisco Monserrat:} Universitat Polit\`ecnica de Val\`encia, Departament de Matem\`atica Aplicada \&  Institut Universitari de Matem\`atica Pura i Aplicada, 46022
València, Spain.\\
\texttt{Carlos Jes\'us Moreno-\'Avila:} Universidad de Extremadura, Escuela Politécnica, Departamento de Matemáticas, 10003, Cáceres, Spain.\\
\texttt{Elvira P\'erez-Callejo:} Universidad de Valladolid, Facultad de Ciencias, Departamento de Álgebra, Análisis Matemático, Geometría y Topología, Valladolid, Spain.}

\email{{\rm Galindo: galindo@uji.es; {\rm Monserrat: framonde@mat.upv.es} ; {\rm Moreno-\'Avila: cjmoravi@unex.es} ; {\rm P\'erez-Callejo: elvira.perez@uva.es}}}

\urladdr{{\rm Galindo: 0000-0002-3908-4462; {\rm Monserrat: 0000-0003-2221-0140} ; {\rm Moreno-\'Avila: 0000-0002-2374-5932} ; {\rm P\'erez-Callejo: 0000-0001-5399-1221} }\color{black}}
\subjclass[2020]{}
\keywords{}

\thanks{Partially funded by MCICIU/AEI/10.13039/501100011033 and by ``ERDF, UE'', grants PID2022-138906NB-C22 and PID2022-138906NB-C21.}

\begin{document}

\begin{abstract}
We give generators for the nef cone and the cone of curves of rational surfaces obtained by blowing-up the complex projective plane at a set of points $\mathcal{B} \cup \mathcal{D}$, where $\mathcal{B}$ is the set of (proper and infinitely near) base points of a pencil associated with a curve having one place at infinity, and  $\mathcal{D}$ is a set of finitely many infinitely near free points on the strict transforms of curves of the pencil. We also prove that, when the pencil is given by an AMS-type curve and $\mathcal{D}$ contains at most two free points on any curve considered, the Cox ring of the obtained surface is finitely generated.
\end{abstract}

\maketitle


\section{Introduction}\label{se:uno}
Birational geometry is a core area of modern algebraic geometry, driving ongoing research and leading to major advances in the understanding of algebraic varieties and their properties. Cones of divisors are very useful objects in this area.

The cone of curves NE$(V)$ of a variety $V$ is one of the main tools used in the minimal model program \cite{Mor, Kaw, BirCasHacMcKer, HacMck, Casc2021}. Mori dream spaces are algebraic varieties with finitely generated  Cox rings Cox$(V)$, and they behave very well under the minimal model program \cite{HuKe}, some references are \cite{LafVel,ArtLaf,ArzDerHauLaf,HauKeiLaf,LafUga}. Other important cones include the ample cone of $V$, which allows us to study embeddings from $V$ into projective spaces, and its closure in the usual topology, the nef cone, Nef$(V)$. Computing this cone is an interesting and difficult task even when considering surfaces.

We are interested in (complex) projective surfaces $Z$. Deciding the nefness of certain divisors on them would take us to prove important conjectures as the Nagata conjecture \cite{Nag}, see the introduction of \cite{FulMur}. In our context, to define Cox$(Z)$, we assume that $Z$ is a smooth surface having a finite set freely generating the Picard group of $Z$. Despite its importance, the classification of smooth complex projective surfaces both with (finite) polyhedral cones of curves and finitely generated Cox rings remain open problems.

For a rational surface $Z$, its anticanonical class, $-K_Z$, plays an important role in the finite generation of Cox$(Z)$, because it holds whenever $-K_Z$ is big \cite{TesVarVel}. Regarding the  nefness of divisors, the degrees of a minimal homogeneous generating set of Cox$(Z)$ are known whenever Cox$(Z)$  is finitely generated and $-K_Z$ is nef \cite{ArtPer}.

It is well known that a projective rational surface $Z$ can be obtained by blowing-up at a configuration (of infinitely near points) over the (complex) projective plane $\mathbb{P}^2$ or a (complex) Hirzebruch surface $\mathbb{F}_\delta$. Although NE$(Z)$ is not always polyhedral in the rational case \cite{CamGonz} (implying Cox$(Z)$ is not finitely generated), recent results show that, when $Z$ arises from blowing-up at a configuration $\mathcal{C}$ over $\mathbb{F}_\delta$, there exist two positive integers $a$ and $b$, that depend on an arrowed dual graph attached to $\mathcal{C}$, such that if $\delta \geq a$, then the cone NE$(Z)$ is polyhedral with easily determined extremal rays. Furthermore, the classes generating Nef$(Z)$ can be explicitly computed. In addition if  $\delta \geq b$, then Cox$(Z)$ is finitely generated \cite{GalMonMor2025-RAC}. See also \cite{RosaFrisMus2017, RosaFrisMus2023,FriLahy2021}.
Significantly less is known when the configuration $\mathcal{C}$ is over $\mathbb{P}^2$  \cite{HarbLahy, GalMon2005, GalMon2005-JA, Ott, GalMon2016}.

An interesting case involves the configuration $\mathcal{B}$ of (proper and infinitely near) base points of a pencil at infinity associated with a curve having one place at infinity and the rational surface $X$ obtained by blowing-up at $\mathcal{B}$. Campillo et al. proved in \cite{CamPilReg-2002} that the cone NE$(X)$ is polyhedral describing its extremal rays. They also proved that the characteristic and nef cones coincide and are regular.

In this paper, we go further and consider configurations $\mathcal{C} =\mathcal{B} \cup \mathcal{D} $, where the points in $\mathcal{D}$ lie on the strict transforms of some curves of the pencil. We prove that the cone of curves NE$(\widetilde{X})$ of the surfaces $\widetilde{X}$ obtained by blowing-up $\mathbb{P}^2$ at $\mathcal{C}$ is polyhedral. Our first main result, Theorem \ref{thm_conos}, supplies explicit generators as convex cones for both NE$(\widetilde{X})$ and Nef$(\widetilde{X})$. Our second main result, Theorem \ref{thm:Cox_ring}, proves that the Cox ring, Cox$(\widetilde{X})$, of surfaces $\widetilde{X}$ linked to Abhyankar-Moh-Suzuki (AMS)-type curves (see Definition \ref{ams_type}) and suitable sets $\mathcal{D}$ is finitely generated (as $\mathbb{C}$-algebra).

Section \ref{sec_preliminares} introduces the necessary notation, recalls the concept of pencil at infinity and the nice structure of the cone of curves of the surface obtained by blowing-up at the configuration $\mathcal{B}$. Section \ref{sec:nef_cone} contains Theorem \ref{thm_conos} and analyzes NE$(\widetilde{X})$ and Nef$(\widetilde{X})$ for a surface $\widetilde{X}$ given by the configuration $\mathcal{C}$. Finally finite generation of the Cox ring in the AMS case is addressed in Section \ref{sec:Cox_ring}, where we show that we consider nonredundant configurations.

\section{Preliminaries}\label{sec_preliminares}

\subsection{Curves and pencils at infinity}\label{subsec_pen.inf}

Let $L$ be the line at infinity in the compactification of the complex affine plane $\mathbb{A}^2$ to the complex projective plane $\mathbb{P}^2.$ An integral projective curve $C$ on $\mathbb{P}^2$ is said to \emph{have one place at infinity} if the intersection $C\cap L$ is a single point $p\in\mathbb{P}^2$ and $C$ is reduced and unibranch at $p$. The point $p$ is called \emph{the point at infinity} of $C$.

Denote $B_{q}(Z)$ the surface obtained by blowing-up a smooth surface $Z$ at a closed point $q \in Z$ and set $E_q$ the resulting exceptional divisor. A {\it constellation} (of infinitely near points) over $Z$ is a finite set of closed points $\mathcal{C}=\{q_j\}_{j=1}^N$ where $q_1=q \in Z_0:=Z$, $q_{j+1} \in Z_j:=\text{Bl}_{q_j} (Z_{j -1})$, $1 \leq j \leq N$, and its origin $q_1$ is the image of $q_j$ under the composition of blowups giving rise to $q_j$ whenever $2 \leq j \leq N$. A {\it configuration} (of infinitely near points) over $Z$ is a finite disjoint union of constellations with origins in $Z$. For indices $1 \leq i, j \leq N$, a point $q_i$ is {\it infinitely near} to the point $q_j$ (denoted $q_i \geq q_j$) if $q_j=q_i$ or if $q_j$ is the image under the composition of blowups providing $q_i$. Moreover, a point $q_i$ is {\it proximate} to the point $q_j$ (denoted $q_i\to q_j$) when $i>j$ and $q_i$ lies on the strict transform of the exceptional divisor $E_{q_j}$ on $Z_i$. A point of the configuration $\mathcal{C}$ is named {\it satellite} when it is proximate to other two points of $\mathcal{C}$. Otherwise, it is called \emph{free}.

Let $C$ be a curve with one place at infinity. Consider the infinite sequence of blowups
\begin{equation}\label{eq_seq_infinity}
    \cdots \rightarrow X_i\xrightarrow{\pi_i} X_{i-1}\rightarrow \cdots \rightarrow X_1 \xrightarrow{\pi_1} X_0:=\mathbb{P}^2,
\end{equation}
where the center of the first blowup $\pi_1$ is the point at infinity $p=:p_{1}$ and, for $i>1$, the morphism $\pi_{i}$ is the blowup of $X_{i-1}$ at the unique point $p_{i}$ lying on the strict transform of $C$ and the exceptional divisor $E_{p_{i-1}}$ (for simplicity, also denoted  $E_{i-1}$) created by $\pi_{i-1}$. Set $m$ the positive integer such that the composition  $\sigma:\pi_m \circ \cdots \circ \pi_1 : X_m \rightarrow \mathbb{P}^2$ gives rise to the minimal embedded resolution of the germ of $C$ at $p$. This implies that the strict transform of $C$ by $\sigma$ at $p_m$ becomes regular and transversal to the exceptional divisor.

Without loss of generality, fix homogeneous coordinates $(\mathcal{X:Y:Z})$ on $\mathbb{P}^2$ and set $\mathcal{Z}=0$ the equation of $L$ and $p=(1:0:0)$ the point at infinity. In the chart $\mathcal{Z}\neq 0$, let $(x,y)$ be coordinates and $f(x,y)=0$ be an equation of $C$. The {\it pencil at infinity} defined by $C$, $\mathcal{P}_C$, is the pencil of projective curves on $\mathbb{P}^2$ whose equations are $\lambda F(\mathcal{X:Y:Z})+\mu \mathcal{Z}^d=0$, where $F(x,y,1)=f(x,y)$, $d$ is the degree of $C$ and $(\lambda:\mu)$ runs over $\mathbb{P}^1$. It is known that those curves where $\lambda\neq 0$ have one place at infinity and are equisingular at $p$. Moreover, there exists a positive integer $n\geq m$ such that the composition of blowups $\pi:\pi_n \circ \cdots \circ \pi_1 : X_n \rightarrow \mathbb{P}^2$ eliminates the base points of the pencil $\mathcal{P}_C$ \cite{Moh}. We call $\mathcal{B}=\{p_i\}_{i=1}^n$ the {\it configuration of base points of the pencil} $\mathcal{P}_C$. Notice that it is in fact a constellation.


The ``topology'' of $\pi$ can be coded in a labelled tree, called \emph{dual graph} of $\pi$ and denoted $\Gamma_\pi.$ Its vertices represent the (strict transforms $\widetilde{E}_i$ on $X_n$ of the) exceptional divisors $E_i$ and are labelled with the number $i$; moreover two vertices labelled $i$ and $j$ are joined by an edge whenever $E_i\cap E_j\neq \emptyset$. Note that sometimes we also represent $E_i$ a strict transform of $E_i$. Let $g$ be the number of vertices of degree $3.$ When $g >0$, we also label these vertices $st_j$, $0 < j \leq g$. For convenience, we set $st_{0}=1$. Vertices of degree $1$ are also labeled $\rho_j$, $0 \leq j \leq g+1$, where $\rho_0 = 1.$ The graph $\Gamma_\pi$ is divided in subgraphs $\Gamma^j$, $1 \leq j \leq g+1$, where $\Gamma^1$ contains the vertices (and attached edges) with labels $1 \leq i \leq st_1$; $\Gamma^j$, $2 \leq j \leq g$, those with labels $st_{j-1} \leq i \leq st_j$ and $\Gamma^{g+1}$ (usually named the tail of $\Gamma_\pi$) the vertices labelled $i$, $st_g \leq i \leq n$. Notice that when $g=0$, $st_1$ is undefined and $\Gamma_\pi$ coincides with its tail. The shape of the dual graph $\Gamma_\pi$ is shown in Figure \ref{Fig_dualgraphdivisorial}.

\begin{center}
\begin{figure}[h]
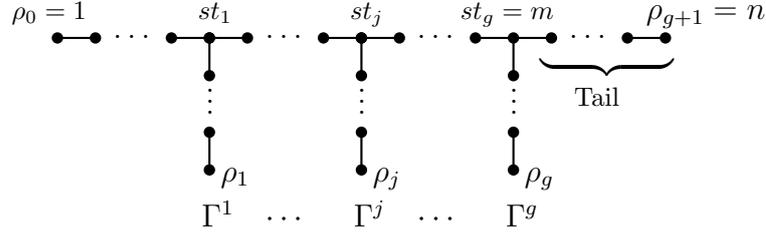


\setlength{\unitlength}{0.5cm}%
\begin{Picture}(0,0)(16,8)
\thicklines


\xLINE(0,6)(1,6)
\Put(0,6){\circle*{0.3}}
\Put(1,6){\circle*{0.3}}
\put(1,6){$\;\;\ldots\;\;$}
\xLINE(3,6)(4,6)
\Put(3,6){\circle*{0.3}}
\Put(4,6){\circle*{0.3}}
\xLINE(4,6)(4,5)
\Put(4,5){\circle*{0.3}}
\Put(3.9,4){$\vdots$}
\xLINE(4,3.5)(4,2.5)
\Put(4,3.5){\circle*{0.3}}
\Put(4,2.5){\circle*{0.3}}
\Put(3.8,1){$\Gamma^1$}

\put(4.3,2.2){$\rho_1$}

\Put(3.7,6.5){\footnotesize $st_1$}
\Put(-1.2,6.5){\footnotesize $\rho_0=1$}



\xLINE(4,6)(5,6)
\Put(4,6){\circle*{0.3}}
\Put(5,6){\circle*{0.3}}
\Put(5,6){$\;\;\ldots\;\;$}
\xLINE(7,6)(8,6)
\Put(7,6){\circle*{0.3}}
\Put(8,6){\circle*{0.3}}
\xLINE(8,6)(8,5)
\Put(8,5){\circle*{0.3}}
\Put(7.9,4){$\vdots$}
\xLINE(8,3.5)(8,2.5)
\Put(8,3.5){\circle*{0.3}}
\Put(8,2.5){\circle*{0.3}}
\Put(7.8,1){$\Gamma^j$}

\put(8.3,2.2){$\rho_j$}

\Put(7.7,6.5){\footnotesize $st_j$}

\Put(5,1){$\;\;\cdots\;\;$}


\xLINE(8,6)(9,6)
\Put(8,6){\circle*{0.3}}
\Put(9,6){\circle*{0.3}}
\put(9,6){$\;\;\ldots\;\;$}
\xLINE(11,6)(12,6)
\Put(11,6){\circle*{0.3}}
\Put(12,6){\circle*{0.3}}
\xLINE(12,6)(12,5)
\Put(12,5){\circle*{0.3}}
\Put(11.9,4){$\vdots$}
\xLINE(12,3.5)(12,2.5)
\Put(12,3.5){\circle*{0.3}}
\Put(12,2.5){\circle*{0.3}}
\Put(11.8,1){$\Gamma^g$}

\put(12.3,2.2){$\rho_g$}

\Put(10.6,6.5){\footnotesize $st_g=m$}

\Put(9,1){$\;\;\cdots\;\;$}


\xLINE(12,6)(13,6)
\Put(13,6){\circle*{0.3}}
\Put(13,6){$\;\;\ldots\;\;$}
\Put(15,6){\circle*{0.3}}
\xLINE(15,6)(16,6)
\Put(16,6){\circle*{0.3}}


\Put(12.7,5.5){$\underbrace{\;\;\;\;\;\;\;\;\;\;\;\;\;\;\;}$}
\Put(13.6,4.2){{\footnotesize Tail}}

\Put(15.5,6.5){$\rho_{g+1}=n$}
%
\end{Picture}
  \caption{Dual graph $\Gamma_\pi$.}
  \label{Fig_dualgraphdivisorial}
\end{figure}
\end{center}

Let $G$ be a curve on $\mathbb{P}^2$. Consider a sequence as in \eqref{eq_seq_infinity} and two positive integers $i$ and $j.$  We use the notation $\varphi_G$ for the corresponding germ of $G$ at the closed point $p$. Additionally, $\varphi_i$ stands for an analytically irreducible  germ at $p$ whose strict transform on $X_i$ is transversal to the exceptional divisor $E_i$ at a general point. We write $\operatorname{mult}_{p_j}(\varphi_i)$ (respectively, $\operatorname{mult}_{p_j}(\varphi_G)$) for the multiplicity of the strict transform of $\varphi_i$ (respectively, $\varphi_G$) at $p_j$. By \cite[Theorem 3.5.3]{Cas}, the following equalities, known as \emph{proximity equalities}, hold:
\begin{equation}\label{eq_proximities_eq}
\operatorname{mult}_{p_i}(\varphi_G)=\sum_{p_j\to p_i}\operatorname{mult}_{p_j}(\varphi_G).
\end{equation}
In addition, $(\varphi,\phi)_p$ denotes the intersection multiplicity at $p$ of two curve germs $\varphi$ and $\phi$. We will frequently apply the Max-Noether formula \cite{Ful2004}, which we recall for the convenience of the reader.

\begin{lemma}\label{lem_for_Noether}
Let $\varphi,\phi$ be two germs of curve on $\mathbb{P}^2$ at $p$ with no irreducible components in common. Then, it holds that
    \begin{equation}\label{eq_formula_Noether}
    (\varphi,\phi)_p=\sum_{q}\operatorname{mult}_q(\widetilde{\varphi})\operatorname{mult}_q(\widetilde{\phi}),
    \end{equation}
    where $q$ runs over all infinitely near to $p$ points that lie on the strict transforms $\widetilde{\varphi}$ and $\widetilde{\phi}$ of both $\varphi$ and $\phi$.
\end{lemma}

We use the value $i_L$ to indicate the last point of the configuration $\mathcal{B}$ through which the strict transform of the line at infinity passes. We often omit the symbol $\tilde{\cdot}$ when considering strict transforms. Observe that $$d:=\deg(C)=(\varphi_L,\varphi_i)_p =\sum_{j=1}^{i_L}\operatorname{mult}_{p_j}(\varphi_L)\operatorname{mult}_{p_j}(\varphi_i) , \, \text{ for } \, m\leq i\leq n.$$

The value semigroup of the curve $C$ at the point $p$ is defined as follows.
$$
S_{C,p}:=\{j\in\mathbb{N} \ | \ j=(\varphi_C,\phi)_p \text{ for some curve germ $\phi$  at $p$}\}.
$$
 The maximal contact values of $C$ at $p,$ $\{\overline{\beta}_i\}_{i=1}^g$, are   $\overline{\beta}_i:=(\varphi_C,\varphi_{\rho_i})_p$ for $0\leq i\leq g$ and they minimally generate $S_{C,p}$. Letting $e_i=\gcd\left(\overline{\beta}_0,\ldots,\overline{\beta}_{i}\right)$ for $0\leq i\leq g$, and $e_{i-1}=e_i\,n_i$ for $1\leq i\leq g$, it holds that $\overline{\beta}_0<\overline{\beta}_1$,  $n_i\overline{\beta}_i<\overline{\beta}_{i+1},$ for $0\leq i\leq g-1,$ and $n_i\,\overline{\beta}_i\in\langle \overline{\beta}_0,\ldots,\overline{\beta}_{i-1} \rangle$ for $1\leq i\leq g$. Moreover, by \cite[Lemma 4]{Bre}, one has that $(\varphi_n,\varphi_{st_j})_p=n_j\overline{\beta}_j$. Further details can be found in \cite{Cas,Zar06}.

The following result will be useful. Item (i) follows from \cite[Proposition 4]{CamPilReg-2002}. Although Item (ii) was previously proven in \cite[Corollary 1]{CamPilReg-2002}, we provide an alternative proof here as our terminology differs significantly from theirs.

\begin{lemma}\label{LemmaCamp}
     Keep the notation introduced before. Then,
     \begin{itemize}
         \item[(i)] $ (\varphi_L,\varphi_n)_p\,(\varphi_L,\varphi_{\rho_{j}})_p>\overline{\beta}_{j},\text{ for } 0\leq j\leq g.$
         \item[(ii)] $ (\varphi_L,\varphi_n)_p\,(\varphi_L,\varphi_r)_p>(\varphi_r,\varphi_n)_p, \, \text{ for }1\leq r<n.$
     \end{itemize}
\end{lemma}
\begin{proof}
    As said, we provide a proof for Item (ii). We proceed in steps based on the value $r$, $1\leq r<n.$ To start, assume that $r\leq i_L$ and consider two subcases, $r<\rho_1$ and $r=\rho_1$. In the first case ($r\leq i_L$ and $r<\rho_1$), the Noether formula implies that  $(\varphi_L,\varphi_r)_p=r$ and $(\varphi_L,\varphi_n)_p>\overline{\beta}_0$. Then, one obtains that
    $$
    (\varphi_L,\varphi_n)_p\,(\varphi_L,\varphi_r)_p-(\varphi_r,\varphi_n)_p=r\,(\varphi_L,\varphi_n)_p -r\overline{\beta}_0>0.
    $$
    Now suppose that $r\leq i_L$ and $r=\rho_1$. This implies that $i_L=r=\rho_1$ and then  $(\varphi_L,\varphi_n)_p=\overline{\beta}_1$. Consequently,
    $$
    (\varphi_L,\varphi_n)_p\,(\varphi_L,\varphi_r)_p-(\varphi_r,\varphi_n)_p=i_L\overline{\beta}_1-\overline{\beta}_1>0,
    $$
    since $i_L\geq 2$.

    The remaining case is $r>i_L$, which we divide into three subcases:

    (a) $i_L=\rho_1$ and $\rho_1<r\leq st_1$;

    (b) either $1\leq j\leq g-1$ and $st_j\leq r \leq \rho_{j+1}$ or $j=1$ and $i_L\leq r<\rho_1$;

    (c) either $1< j\leq g$ and $\rho_j< r \leq st_{j},$ or $\rho_1<r\leq st_1$ and $i_L<\rho_1.$\\

    {\it Subcase (a):} By the Noether formula, $(\varphi_L,\varphi_n)_p=\overline{\beta}_1$ and $(\varphi_L,\varphi_r)_p>(i_L-1)n_1\geq n_1$. In addition, by \cite[Lemma 4]{Bre}, $(\varphi_n,\varphi_{st_1})_p=n_1\overline{\beta}_1$. As a consequence, it follows that
    $$
    (\varphi_L,\varphi_n)_p\,(\varphi_L,\varphi_r)_p-(\varphi_r,\varphi_n)_p>n_1\overline{ \beta}_1-n_1\overline{\beta}_1=0.
    $$

    {\it Subcase (b):} If $r=\rho_{j+1}$, the inequality follows by Item (i). When $r\neq \rho_{j+1},$ again by Item (i) it holds
    $$
    (\varphi_L,\varphi_n)_p\,(\varphi_L,\varphi_r)_p-(\varphi_r,\varphi_n)_p\geq (\varphi_L,\varphi_n)_p\,(\varphi_L,\varphi_{\rho_{j+1}})_p-\overline{\beta}_{j+1}>0.
    $$

    {\it Subcase (c):} Consider the sequence of blowups $\pi_r \circ \cdots \circ \pi_1 : X_r \rightarrow \mathbb{P}^2$ and the corresponding maximal contact values for this sequence. That is, the values $\overline{\beta}_k^r:=(\varphi_r,\varphi_{\rho_k})_p$ for $1\leq k\leq j$. Denote $e_k^r:=\gcd(\overline{\beta}_0^r,\ldots,\overline{\beta}_{k}^r)$ for $0\leq k\leq j$ and $e_{k-1}^r:=n_k^r\,e_{k}^r$ for $1\leq k\leq j$. Then, $(\varphi_r,\varphi_r)_p=n_j^r\overline{\beta}_j^r$ by \cite[Lemma 4]{Bre} and then, by the Noether formula, one has that
    $$
    (\varphi_n,\varphi_r)_p=e_j\,(\varphi_r,\varphi_r)_p= e_j\,n_j^r\overline{\beta_j}^r\leq n_j^r\overline{\beta}_j \text{ and } (\varphi_L,\varphi_r)_p= n_j^r(\varphi_L,\varphi_{\rho_j})_p.
    $$
    Therefore, we deduce:
    \begin{equation*}
    \begin{array}{rl}
        (\varphi_L,\varphi_n)_p\,(\varphi_L,\varphi_r)_p-(\varphi_n,\varphi_r)_p\geq n_j^r\,(\varphi_L,\varphi_n)_p\, (\varphi_L,\varphi_{\rho_j})_p-n_j^r\overline{\beta}_j  > 0,
       \end{array}
    \end{equation*}
    where the second inequality follows from Item (i).
\end{proof}

\subsection{Surfaces obtained by eliminating base points of pencils at infinity.}

Maintaining the notation from Subsection \ref{subsec_pen.inf}, let $Z$ be the rational projective surface obtained by a finite sequence of blowups $\pi: Z \rightarrow \mathbb{P}^2$. Set $\operatorname{NS}(Z)$ the Nerón-Severi group of $Z$ and  $\operatorname{NS}_\mathbb{R}(Z):=\operatorname{NS}(Z)\otimes\mathbb{R}$. The symbol $\cdot$ denotes the intersection product within these spaces. As mentioned in the introduction, let $\operatorname{NE}(Z)$ (respectively, $\overline{\operatorname{NE}}(Z),\,\operatorname{Nef}(Z)$ and $\operatorname{Cox}(Z)$) be the cone of curves (respectively, the pseudoeffective cone or Mori cone, the nef cone and the Cox ring) of $Z$. We usually denote by $\widetilde{E}$ (respectively, $E^*$) the strict (respectively, total) transform on $Z$ of a divisor $E$ involved in the sequence $\pi$.

The following result due to Campillo et al. \cite{CamPilReg-2002} studies the case $s=n$ (minimal sequence of blowups eliminating the base points of $\mathcal{P}_C$, $C$ being  a curve having one place at infinity).

\begin{theorem}\label{thm_Campillo}
    Let $X$ be the rational projective surface obtained by the composition of blowups $\pi:\pi_n \circ \cdots \circ \pi_1 : X:=X_n \rightarrow \mathbb{P}^2$ that eliminates the base points of the pencil $\mathcal{P}_C$. Then, the cone of curves $\operatorname{NE}(X)$ is generated by the classes in $\mathrm{NS}_{\mathbb{R}}(X)$ of the divisors in the set $\{\widetilde{L},\widetilde{E}_1,\ldots,\widetilde{E}_n\}$.
\end{theorem}

Theorem \ref{thm_Campillo} and \cite{GalMon2005-JA} prove that, for $X$ as before,
the nef cone $\operatorname{Nef}(X)$ and the Cox ring $\operatorname{Cox}(X)$ are finitely generated.

\section{The nef cone of a surface given by configurations extending the pencils at infinity}\label{sec:nef_cone}

Let $C$ be a curve with one place at infinity and $\mathcal{P}_C$ the pencil of plane curves on $\mathbb{P}^2$ associated to $C$. Denote by $\mathcal{B}=\{p_i\}_{i=1}^n$ the configuration of base points of $\mathcal{P}_C$, and by $X_n$ the surface obtained by blowing-up at the points in $\mathcal{B}$. For simplicity, set $p:=p_{1}$. Note that the strict transform on $X=X_n$ of any curve in $\mathcal{P}_C$ with support different from $L$ passes through a unique free point of the exceptional divisor $E_n$ \cite[Theorem 1]{CamPilReg-2002} and its self-intersection vanishes. As a consequence, if we consider the surface $X_{n+1}$ obtained by blowing-up $X_n$ at a free point $q\in E_n$, then, the class of the strict transform on $X_{n+1}$ of any curve in $\mathcal{P}_C$ as before spans an extremal ray of the cone of curves $\operatorname{NE}(X_{n+1})$, see \cite[Lemma 1.22]{KolMor}.

Our aim in this section is to study the behavior of the nef cone of rational surfaces obtained by blowing-up $\PP^2$ at configurations $\conf$ as follows: $\conf=\mathcal{B}\,\cup\, \mathcal{D},$ where $\mathcal{B}$ is the configuration of base points of $\mathcal{P}_C$ and $\mathcal{D}=\bigcup_{j=1}^k \{p_{j\,\ell_j}\}_{\ell_j=1}^{s_j}$ is a configuration as we are going to describe. For a start, $p_{j\,1}\to p_n,$ for $1\leq j\leq k,$ and $p_{j\,i+1}\to p_{j\,i},$ for $1\leq i\leq s_{j}-1$ when $s_j>1$. The points $p_{j\,1}$ are free and arbitrarily chosen on $E_n$. Let $C_j$ be the curve  of $\mathcal{P}_C$ whose strict transform goes through the point $p_{j\,1}\in E_n$, then, the points $p_{j\,\ell_j}, 1\leq \ell_j \leq s_j$, are finitely many infinitely near free points through which the strict transforms of $C_j$ pass. For convenience, we re-index the points in $\mathcal{B}=\{p_{0\,\ell_0}\}_{\ell_0=1}^{s_0}$ as $p_{0\,\ell_0}:=p_{\ell_0}$ and $s_0=n$. Therefore, the configuration $\mathcal{C}$ can be written as:
\begin{equation}\label{eq_conf}\mathcal{C}=\{q_u\}_{u=1}^s=\bigcup_{j=0}^k\left[\{p_{j\,\ell_j}\}_{\ell_j=1}^{s_j}\right], \, \text{ where } s=\sum_{j=0}^ks_j.\end{equation}

From now on, we adapt the notation introduced in Section \ref{sec_preliminares} according to the indices of the points in $\mathcal{C}.$ Let $\widetilde{X}$ be the surface obtained by blowing-up $\PP^2$ at the points in $\mathcal{C}.$ For our purposes, it is useful to consider the strongly convex cone $P$ of $\mathrm{NS}_{\mathbb{R}}(\widetilde{X})$ generated by the following set of classes:
\begin{equation}\label{eq_original_cone}
    \Delta_P = \Big\{[\widetilde{L}], [\widetilde{C}_1], \ldots, [\widetilde{C}_k], [\widetilde{E}_{0\,1}], \ldots[\widetilde{E}_{0\,s_0}], \ldots, [\widetilde{E}_{k\,1}],\ldots, [\widetilde{E}_{k\,s_k}]\Big\}.
\end{equation}
Proposition \ref{prop_dual} will be our first result in this section and it will describe the generators of the dual cone of $P$:
$$
P^\vee:=\big\{[D]\in\text{NS}_{\mathbb{R}}(\widetilde{X}) \ | \ [D]\cdot [G]\geq 0, \text{ for all }[G]\in P\big\}.
$$
To state our result, it is convenient to introduce {\it two} families of divisors and an example that illustrates the nature of the divisors.

Regarding our {\it first} family of divisors, we set
\begin{equation}\label{eq_div_D_0_m0}
\begin{array}{cl}
    D_{0\,0}&:=E_{0\,0}^*,\\
D_{0\,m(0)}&:=(\varphi_L,\varphi_{p_{0\,m(0)}})_p\,E_{0\,0}^* -\sum_{\ell_0=1}^{m(0)}\operatorname{mult}_{p_{0\,\ell_0}}(\varphi_{0\,m(0)})E_{0\,\ell_0}^*,\\
&\text{for every index }m(0),\, 1\leq m(0)\leq s_0=n,
\end{array}
\end{equation}
and, for indices $1\leq j\leq k$ and $1\leq m(j)\leq s_j$,
$$
D_{j\,m(j)}:=(\varphi_L,\varphi_{p_{j\,m(j)}})_p\,E_{0\,0}^* -\sum_{\ell_0=1}^{s_0}\operatorname{mult}_{p_{0\,\ell_0}}(\varphi_{0\,m(j)})E_{0\,\ell_0}^*-\sum_{\ell_j=1}^{m(j)}\operatorname{mult}_{p_{j\,\ell_j}}(\varphi_{j\,m(j)})E_{j\,\ell_j}^*.
$$

For convenience, we assume that $0\leq m(0)\leq s_0=n$. In particular, $D_{0\,m(0)}=D_{0\,0}=E_{0\,0}^*$. Moreover, for $1\leq j \leq k$, $D_{j\,s_j}$ (respectively, $D_{j\, m(j)}$) is numerically equivalent to the strict transform of $C_j$ on the surface $\widetilde{X}$ (respectively, $\widetilde{C}_j|_{j\,m(j)}$ on $X_{j\,m(j)}$). Furthermore, by definition (and using Lemma \ref{LemmaCamp} where necessary), the following equalities and inequalities hold:
\begin{equation}\label{eqn_Ds_rela}
\begin{array}{ll}
    D_{0\,m(0)}\cdot D_{j\,m(j)}=D_{0\,m(0)}\cdot D_{0\,s_0}>0,&\text{ if }1\leq j\leq k \text{ and } 0\leq m(0)\leq s_0-1,\\[2mm]
    D_{0\,m(0)}^2= D_{0\,m(0)}\cdot  D_{0\,m(0)}=0,& \text{ if }m(0)=s_0,\\[2mm]
    D_{j\,m(j)}\cdot D_{j\,m'(j)}=-\min\{m(j),m'(j)\},&\text{ if }1\leq j\leq k\text{ and } 1\leq m(j),m'(j)\leq s_j,\\[2mm]
    D_{j_1\,m(j_1)}\cdot D_{j_2\,m(j_2)}=0,&\text{ if }1\leq j_1\neq j_2\leq k.\\[2mm]
\end{array}
\end{equation}
Next, we introduce the {\it second} family of divisors. Consider the set of subsets
$$\mathcal{I}:=\big\{J=\{j_1<\cdots<j_c\}\subseteq\{1,\ldots,k\} \ | \ 1\leq c\leq k\big\}.$$
Given  $J\in\mathcal{I},$ define $\mathcal{S}_J:=\{1,\ldots,s_{j_1}\}\times \cdots \times \{1,\ldots,s_{j_c}\}$. An element in $\mathcal{S}_J$ is written as $\bar{\ell}_J=(\ell_{j_1},\ldots,\ell_{j_c}).$ In addition, setting  $\Pi_{\bar{\ell}_J}:=\ell_{j_1}\cdots \ell_{j_c}$ and $a_{m(0)}:=D_{0\,m(0)}\cdot D_{0\, s_0}$, we define, for $0\leq m(0)\leq s_0-1$, the divisor $D_{0\,m(0)}^{J,\,\bar{\ell}_J}$ as
\begin{equation}\label{eq_div_J_l_J}
D_{0\,m(0)}^{J,\,\bar{\ell}_J}:=a_{m(0)}\Bigg[\sum_{j_t\in J} \left(\dfrac{\Pi_{\bar{\ell}_J}}{\ell_{j_t}}\cdot D_{j_t\,\ell_{j_t}}\right)\Bigg]+\Pi_{\bar{\ell}_J}\cdot D_{0\,m(0)}, \text{ where } 1\leq t\leq c.
\end{equation}
Note that, when $m(0)=s_0$, $a_{m(0)}=0$ by \eqref{eqn_Ds_rela}, and the divisor $D_{0\,m(0)}^{J,\,\bar{\ell}_J}$ is a multiple of $D_{0\,s_0}.$

\medskip

We now present an example to clarify the expressions for the above given divisors. In this example, we consider a curve having one place at infinity whose singularity at infinity is a cusp. This example will be continued in Example \ref{exa_cono_nef_curvas}, where we describe the generators of the nef cone and the cone of curves of the associated surface $\widetilde{X}$.

\begin{example}\label{ex_cusp_not}
We use the coordinates established at the beginning of Section \ref{sec_preliminares}. Let $C$ be the curve  on $\mathbb{P}^2$ given locally by the equation $x-y^3=0$ in the affine chart $\mathcal{Z}\neq0$. Note that $C$ is defined by the local equation $v^2-u^3=0$ in coordinates $(u,v)$ at the point at infinity. As before, $L$ is the line at infinity. The pencil associated with $C$ is generated by $C$ and $3L$; The elimination of its base points requires six blowups. Thus, following the previous notation, $n=s_0=6$ and $\mathcal{B}=\{p_{0\,\ell_0}\}_{\ell_0=1}^6$ where $p_{0\,\ell_0}\to p_{0\,\ell_0-1}$ for all $\ell_0\in\{2,\ldots,6\}$ and $p_{0\,3}\to p_{0\,1}$. On the surface $X_6$ we take two free points $p_{1\,1},p_{2\,1}\in E_{0\,6}$ through which the strict transforms of two curves $C_1$ and $C_2$ of $\mathcal{P}_C$ pass. Additionally, we consider the free point $p_{1\,2}$ proximate to $p_{1\,1}$ through which $\widetilde{C}_1$ passes. The configuration we consider to define a surface $\widetilde{X}$ is  $\conf=\{q_i\}_{i=1}^{9}=\{p_{0\,\ell_0}\}_{\ell_0=1}^6\cup\{p_{1\,\ell_1}\}_{\ell_1=1}^2\cup\{p_{2\,1}\}$. The dual graph of the composition of blowups given by $\conf$ is shown in Figure \ref{fig_cusp}. In the figure, the vertices corresponding to $E_{p_{0\,\ell_0}},$ for $1\leq \ell_0\leq 5,$  (respectively, $E_{p_{0\,6}}$, $E_{p_{j\,\ell_j}},$ for $j=1,2$) are coloured in black (respectively, red, blue). Under the notation described above, $n=s_0=6,\,k=2,\,s_1=2,\,s_2=1$ and $s=9$.

\begin{center}
\begin{figure}[h!]
    \centering
    \definecolor{rvwvcq}{rgb}{0.08235294117647059,0.396078431372549,0.7529411764705882}
\definecolor{ffqqqq}{rgb}{1,0,0}
\begin{tikzpicture}[line cap=round,line join=round,>=triangle 45,x=1.5cm,y=1.5cm]
\clip(-0.5,-0.5) rectangle (6.5,2.5);
\draw [line width=1pt] (0,1)-- (5,1);
\draw [line width=1pt] (4,1)-- (5,2);
\draw [line width=1pt] (5,2)-- (6,2);
\draw [line width=1pt] (1,1)-- (1,0);
\draw (-0.2,1.5) node[anchor=north west] {$E_{p_{0\,1}}$};
\draw (1,0.4) node[anchor=north west] {$E_{p_{0\,2}}$};
\draw (0.8,1.5) node[anchor=north west] {$E_{p_{0\,3}}$};
\draw (1.8,1.5) node[anchor=north west] {$E_{p_{0\,4}}$};
\draw (2.8,1.5) node[anchor=north west] {$E_{p_{0\,5}}$};
\draw (3.8,0.9) node[anchor=north west] {$E_{p_{0\,6}}$};
\draw (4.8,2.5) node[anchor=north west] {$E_{p_{1\,1}}$};
\draw (5.8,2.5) node[anchor=north west] {$E_{p_{1\,2}}$};
\draw (4.8,1.5) node[anchor=north west] {$E_{p_{2\,1}}$};
\begin{scriptsize}
\draw [fill=black] (0,1) circle (2.5pt);
\draw [fill=black] (1,1) circle (2.5pt);
\draw [fill=black] (1,0) circle (2.5pt);
\draw [fill=black] (2,1) circle (2.5pt);
\draw [fill=black] (3,1) circle (2.5pt);
\draw [fill=ffqqqq] (4,1) circle (2.5pt);
\draw [fill=rvwvcq] (5,1) circle (2.5pt);
\draw [fill=rvwvcq] (5,2) circle (2.5pt);
\draw [fill=rvwvcq] (6,2) circle (2.5pt);
\end{scriptsize}
\end{tikzpicture}

    \caption{Dual graph of $\conf$}
    \label{fig_cusp}
\end{figure}
\end{center}

According to its definition, the explicit expressions of the divisors $D_{j\,m(j)}$, $j\in\{0,1,2\}$, in our {\it first} family are:
	\begin{enumerate}
		\item $D_{0\,0}=E_{0\,0}^*$.
		\item $D_{0\,1}=E_{0\,0}^*-E_{0\,1}^*$.
		\item $D_{0\,2}=2E_{0\,0}^*-E_{0\,1}^*-E_{0\,2}^*$.
		\item $D_{0\,3}=3E_{0\,0}^*-2E_{0\,1}^*-E_{0\,2}^*-E_{0\,3}^*$.
		\item $D_{0\,4}=3E_{0\,0}^*-2E_{0\,1}^*-E_{0\,2}^*-E_{0\,3}^*-E_{0\,4}^*$.
		\item $D_{0\,5}=3E_{0\,0}^*-2E_{0\,1}^*-E_{0\,2}^*-E_{0\,3}^*-E_{0\,4}^*-E_{0\,5}^*$.
		\item $D_{0\,6}=3E_{0\,0}^*-2E_{0\,1}^*-E_{0\,2}^*-E_{0\,3}^*-E_{0\,4}^*-E_{0\,5}^*-E_{0\,6}$.
		\item $D_{1\,1}=3E_{0\,0}^*-2E_{0\,1}^*-E_{0\,2}^*-E_{0\,3}^*-E_{0\,4}^*-E_{0\,5}^*-E_{0\,6}-E_{1\,1}^*$.
		\item $D_{1\,2}=3E_{0\,0}^*-2E_{0\,1}^*-E_{0\,2}^*-E_{0\,3}^*-E_{0\,4}^*-E_{0\,5}^*-E_{0\,6}-E_{1\,1}^*-E_{1\,2}^*$.
		\item $D_{2\,1}=3E_{0\,0}^*-2E_{0\,1}^*-E_{0\,2}^*-E_{0\,3}^*-E_{0\,4}^*-E_{0\,5}^*-E_{0\,6}-E_{2\,1}^*$.
	\end{enumerate}

\medskip

Finally let us show the expressions for the {\it second} family of divisors $D_{0\,m(0)}^{J,\bar{\ell}_J}$. In this case,
$$\mathcal{I}=\{\{1\},\{2\},\{1,2\}\}\text{ and } a_0=3,\,a_1=1,\,a_2=3,\,a_3=3,\,a_4=2\text{ and }a_5=1,$$
and, according to the chosen set $J$, our divisors are the following ones:
    \begin{enumerate}
		\item For $J=\{1\}$, $\mathcal{S}_{J}=\{1,2\}$ and then
        $$D^{\{1\},\bar{\ell}_1}_{0\,m(0)}=a_{m(0)}D_{1,\bar{\ell}_1} +D_{0\, m(0)},\, 1\leq \bar{\ell}_1\leq 2,\,1\leq m(0)\leq 5,$$
        where we set $\bar{\ell}_{\{1\}}=\bar{\ell}_1$ for convenience.
		\item For $J=\{2\}$, $\mathcal{S}_J=\{1\}$, $\bar{\ell}_J=1$ and $$D^{\{2\},1}_{0\,m(0)}=a_{m(0)}D_{2,1} +D_{0\, m(0)}, 1\leq m(0)\leq 5.$$
		\item It remains to consider $J=\{1,2\}$, then $\mathcal{S}_J=\{1,2\}\times\{1\}$, $ \bar{\ell}_J\in\{(1,1),(2,1)\}$, and the divisors are:
        \medskip
		\begin{enumerate}
			\item $D^{\{1,2\}\,(1,1)}_{0\,m(0)}=a_{m(0)}(D_{1\,1}+D_{2\,1})+D_{0\,m(0)}$,
			\item $D^{\{1,2\}\,(2,1)}_{0\,m(0)}=a_{m(0)}(D_{1\,2}+2D_{2\,1})+2D_{0\,m(0)}$,
		\end{enumerate}
        where $m(0)\in\{1,\ldots,5\}$.
	\end{enumerate}
\end{example}	

Recall that $P$ is the convex cone of $\mathrm{NS}_\R(\widetilde{X})$ generated by the classes in (\ref{eq_original_cone}) and $P^\vee$ is its dual cone. The following result proves that we have found generators for $P^\vee$.

\begin{proposition}\label{prop_dual}
        Keep the notation introduced in this section. The dual cone of $P$, $P^\vee$, is generated by the classes in $\operatorname{NS}_{\mathbb{R}}(\widetilde{X})$ of the following divisors:
        \begin{enumerate}
            \item[$\bullet$]  $D_{0\,m(0)}, \text{ where }m(0)\in\{0,\ldots,s_0\}$, and
            \item[$\bullet$]  $D_{0\,m(0)}^{J,\bar{\ell}_J}, \text{ where }J\in\mathcal{I}, \bar{\ell}_J\in\mathcal{S}_J \text{ and } m(0)\in\{0,\ldots,s_0-1\}.$
        \end{enumerate}
\end{proposition}

\begin{proof}
Section 1.2 of \cite{Ful2} proves that, to find the generators of $P^\vee$, it suffices to consider every linear subspace $\mathcal{H}$ of codimension one generated by elements in $\Delta_P$ and check if there exists an element in $P^\vee$ generating the orthogonal space of $\mathcal{H}$, $\mathcal{H}^\perp$, with respect to the intersection product.

Let \[
D=d_{0\,0}E_{0\,0}^*-\sum_{j=0}^{k}\sum_{\ell_j=1}^{s_j}d_{j\,\ell_j}E_{j\,\ell_j}^*\]
be a divisor whose class belongs to $P^\vee\subset\operatorname{NS}_{\mathbb{R}}(\widetilde{X})$ with $d_{j\,\ell_j}\in\mathbb{Z}_{\geq 0}$. The definition of $P^\vee$ forces the coefficients of $D$ to satisfy certain inequalities. First, $D\cdot \widetilde{E}_{j\,\ell_j}\geq 0,$ $0\leq j\leq k \text{ and }1\leq \ell_j\leq s_j$, which implies
\begin{equation}\label{eqn_DEi}
\begin{array}{cl}
   d_{0\,\ell_0}-\sum_{p_{0\,k}\to p_{0\,\ell_0}}d_{0\,k}\geq 0, & 0\leq \ell_0\leq s_0-1,\\[2mm] d_{0\,s_0}-\sum_{j=1}^kd_{j\,1}\geq 0,\\[2mm]
   d_{j\,\ell_j}-d_{j\,\ell_{j}+1}\geq 0, & 1\leq j\leq k,\\[2mm]
   d_{j\,s_j}\geq 0, & 1\leq j\leq k.
   \end{array}
\end{equation}
Second,
\begin{equation}\label{eqn_DL}
d_{0\,0}-\sum_{i=1}^{i_L} d_{0\,i_L}\geq 0,
\end{equation}
because $D\cdot \widetilde{L}\geq 0$. Finally, it must hold $D\cdot \widetilde{C}_j\geq 0,\,1\leq j\leq k,$ which, combined with the numerical equivalence $\widetilde{C}_j\sim D_{j\,s_j},$ shows that
\begin{equation}\label{eqn_DC}
d_{0\,0}\,(\varphi_L,\varphi_{j\,s_j})_p-\sum_{\ell_0=1}^{s_0}d_{0\,\ell_0}\operatorname{mult}_{p_{0\,\ell_0}}(\varphi_{j\,s_j})-\sum_{\ell_{j}=1}^{s_j}d_{j\,\ell_j}\geq 0.
\end{equation}

Now we compute the possible values for $d_{j\,\ell_j}$. Recall that $\operatorname{NS}_{\mathbb{R}}(\widetilde{X})$ is a $\mathbb{R}$-vector space of dimension $s+1.$ Let $A$ be a generating set of a linear space $\mathcal{H}= \mathrm{span}(A)$ as above, that is, $A$ is a subset of cardinality $s$ of the set of generators $\Delta_P$  of $P$ given in (\ref{eq_original_cone}). Clearly, $A= \Delta_P\setminus B$, where $B\subseteq \Delta_P$ is a set of $k+1$ elements. Then, each element in $A$ yields equalities whose expressions are of the type of those given in \eqref{eqn_DEi}, \eqref{eqn_DL} or \eqref{eqn_DC} but replacing $\geq$ with $=$.

We now examine the expressions for the divisors $D$ according to the considered sets $A$. Note that if $[\widetilde{E}_{0\,s_0}]\in B$ and $D\cdot \widetilde{E}_{0\,s_0}>0$, then $\mathrm{span}(A)^\perp=\mathrm{span}([D_{0\,s_0}])$. We, therefore, focus on the subsets $B$ of cardinality $k+1$
that not contain $[\widetilde{E}_{0\,s_0}]$.

First, assume $$A=\Delta_P\setminus\left\{\left[\widetilde{L}\right],\left[\widetilde{C}_1\right],\ldots,\left[\widetilde{C}_k\right]\right\}.$$
Solving the resulting linear system of equations, we find that $\mathrm{span}(A)^\perp=\mathrm{span}\left([D_{0\,0}]\right)$. Here, the orthogonal space is computed by solving a linear system of equations whose solutions must also strictly satisfy those inequalities in \eqref{eqn_DEi}, \eqref{eqn_DL} and \eqref{eqn_DC} that are not equalities. This is performed in order to guarantee that our solutions are in $P^\vee$.

\medskip

Similarly, for $1\leq m(0)\leq s_0 -1$,  let
$$
A=\Delta_P \setminus\left\{\left[\widetilde{E}_{0\,m(0)}\right],\left[\widetilde{C}_1\right],\ldots,\left[\widetilde{C}_k\right]\right\}.
$$

Then, $\mathrm{span}(A)^\perp=\mathrm{span}\left([D_{0\,m(0)}]\right),$ where $\mathrm{span}(A)^\perp$ should be understood in the above sense.

\medskip

Finally, let $J\in\mathcal{I}$ and consider the set $\mathcal{S}_J$ defined before Example \ref{ex_cusp_not}. For each $\bar{\ell}_J=(\ell_{j_1},\ldots,\ell_{j_c})\in\mathcal{S}_J$, we consider the linear subspaces generated by $A_1=\Delta_P \,\setminus B_1$, where
$$
B_1=\left\{\left[\widetilde{L}\right]\right\}\cup\left\{\left[\widetilde{C}_t\right]\right\}_{t\not\in J}\cup\left\{\left[\widetilde{E}_{j_t\,\ell_{j_t}}\right]\right\}_{t=1}^c,
$$
and, for each $0\leq m(0)\leq s_0-1,$ those generated by $ \,A_{2\,m(0)}=\Delta_P\,\setminus B_{2\,m(0)},$ where
$$
B_{2\,m(0)}=\left\{\left[\widetilde{E}_{0\,m(0)}\right]\right\}\cup\left\{\left[\widetilde{C}_t\right]\right\}_{t\not\in J}\cup\left\{\left[\widetilde{E}_{j_t\,\ell_{j_t}}\right]\right\}_{t=1}^c.
$$
Considering the related systems of linear equations, we obtain $$\mathrm{span}(A_1)^\perp=\mathrm{span}\left(\left[D_{0\,0}^{J\,\bar{\ell}_J}\right]\right) \text{ and } \mathrm{span}(A_{2\,m(0)})^\perp=\mathrm{span}\left(\left[D_{0\,m(0)}^{J\,\bar{\ell}_J}\right]\right).$$

To conclude, the remaining choices for $A$ do not provide additional generators of $P^\vee$, as the corresponding systems of linear equations yield only the trivial solution.

\end{proof}

The above given generators of $P^\vee$ satisfy a noteworthy property as stated in the following result.

\begin{proposition}\label{prop_autointer}
    Keep the notation introduced above. The generators of the dual cone $P^\vee$ provided in Proposition \ref{prop_dual} have nonnegative self-intersection.
\end{proposition}

\begin{proof}

    We start by observing that, by the second equality in \eqref{eqn_Ds_rela}, $(D_{0\,s_0})^2=0$. Consequently, by \cite[Lemma 2]{GalMon2016}, $(D_{0\,m(0)})^2\geq 0$ for $0\leq m(0) \leq s_0-1$. To conclude, we show that the divisors $D_{0\,m(0)}^{J,\bar{\ell}_J}$ also satisfy this condition. Indeed,
    \begin{align*}
\left(D_{0\,m(0)}^{J\,\bar{\ell}_J}\right)^2=&\left(a_{m(0)}\Bigg[\sum_{j_t\in J} \left(\dfrac{\Pi_{\bar{\ell}_J}}{\ell_{j_t}} D_{j_t\,\ell_{j_t}}\right)\Bigg]+\Pi_{\bar{\ell}_J}\cdot D_{0\,m(0)}\right)^2\\
        =&-a_{m(0)}^2\sum_{j_t\in J}\left(\dfrac{\Pi_{\bar{\ell}_J}}{\ell_{j_t}}\right)^2\ell_{j_t}+\left(\Pi_{\bar{\ell}_J}\cdot D_{0\,m(0)}\right)^2 \\ & + 2a_{m(0)}\,\Pi_{\bar{\ell}_J}\sum_{j_t\in J}\dfrac{\Pi_{\bar{\ell}_J}}{\ell_{j_t}}D_{0\,s_0}\cdot D_{0\,m(0)}\\
        =&\, (\Pi_{\bar{\ell}_J})^2\left(D_{0\,m(0)}^2+ a_{m(0)}^2\sum_{j_t\in J}\frac{1}{\ell_{j_t}}\right)\geq 0,
    \end{align*}
    where the second equality follows from  \eqref{eqn_Ds_rela} and the last one comes from the definition of the value $a_{m(0)}=D_{0\,m(0)}\cdot D_{0\,s_0}.$
\end{proof}

To conclude this section we state and prove our first main result in this paper. It gives generators for the nef cone and the cone of curves of a surface obtained from a sequence of blowups determined by a configuration as that given in  \eqref{eq_conf}. This proves, as a consequence, that both cones are finitely generated.

\begin{theorem}\label{thm_conos}
    Let $C$ be a curve with one place at infinity and $\mathcal{P}_C$ the pencil associated to $C$. Let  $\mathcal{C}=\bigcup_{j=0}^k\left[\{p_{j\,\ell_j}\}_{\ell_j=1}^{s_j}\right]$ be a configuration as in \eqref{eq_conf} and let $\widetilde{X}$ be the rational surface obtained by the sequence of blowups $\pi:\widetilde{X}\to\mathbb{P}^2$ given by $\mathcal{C}$. Then, the cone of curves $\operatorname{NE}(\widetilde{X})$ is generated by the classes of the strict transforms  on $\widetilde{X}$ of: the line at infinity, the  curves $C_1, \ldots, C_k$ in $\mathcal{P}_C$ whose strict transforms  go through $p_{1\,1},\ldots,p_{k\,1}\in\mathcal{C}$ and the exceptional divisors of $\pi$. Additionally, the nef cone $\operatorname{Nef}(\widetilde{X})$ is generated by the classes of the divisors considered in the statement of Proposition \ref{prop_dual}.
\end{theorem}

\begin{proof}
    Keep the above notation. Recall that $P$ is the convex cone of $\operatorname{NS}_{\mathbb{R}}(\widetilde{X})$ spanned by the classes \[
    [\widetilde{L}], [\widetilde{C}_1], \ldots, [\widetilde{C}_k], [\widetilde{E}_{0,1}], \ldots[\widetilde{E}_{0,s_0}], \ldots, [\widetilde{E}_{k,1}], \ldots,[\widetilde{E}_{k,s_k}]\]
    and, by Proposition \ref{prop_dual}, $P^\vee\subseteq\operatorname{NS}_{\mathbb{R}}(\widetilde{X})$ is spanned by the classes of the divisors $D_{0\,m(0)}$, $0\leq m(0)\leq s_0$, defined in \eqref{eq_div_D_0_m0}  and the divisors $D_{0\,m(0)}^{J,\bar{\ell}_J}$ defined in \eqref{eq_div_J_l_J}. Set $H$ an ample divisor on $\widetilde{X}$ and consider the convex cone
    \[
    Q(\widetilde{X}):=\{x\in\operatorname{NS}_{\mathbb{R}}(\widetilde{X})\,|\,x^2\geq 0 \text{ and }H\cdot x\geq 0\}.
    \]
    Let us prove the first statement. By Proposition \ref{prop_autointer}, $D\in Q(\widetilde{X})$ for all $D\in P^\vee$. Letting $\overline{\rm NE}(\widetilde{X})$ be the topological closure of ${\rm NE}(\widetilde{X})$ and $+$ the Minkowski sum, the inclusion
    $P+P^\vee\subseteq \overline{\rm NE}(\widetilde{X})$ holds because the generators of $P$ are curves and $P^\vee\subseteq Q(\widetilde{X})\subseteq \overline{\rm NE}(\widetilde{X})$ by \cite[Lemma 1.20]{KolMor}. Moreover, any irreducible curve on $\widetilde{X}$ other than the generators of $P$, must belongs to $P^\vee$. Thus, \[
    {\rm NE}(\widetilde{X})\subseteq P+P^\vee\subseteq\overline{\rm NE}(\widetilde{X}).
    \]
    Now, taking topological closures, $P=\bar{P}$ and $P^\vee=\overline{P^\vee}$
    because the involved cones are finitely generated, the equality $\overline{\rm NE}(\widetilde{X})=P+P^\vee$ holds.
    Note that $Q(\widetilde{X})\subseteq(P^\vee)^\vee=P$ by \cite[Lemma 1]{GalMon2016} and Proposition \ref{prop_autointer}. Therefore $P^\vee\subseteq P$ and $\overline{\rm NE}(\widetilde{X})=P$. Finally, since $P$ is finitely generated by effective classes, it holds that $P={\rm NE}(\widetilde{X})$.

    The second statement is straightforward from first one as ${\rm NE}(\widetilde{X})$ is finitely generated and thus  $\operatorname{Nef}(\widetilde{X})= P^\vee$.
\end{proof}

\begin{example}\label{exa_cono_nef_curvas}
    Consider the configuration $\conf$ in Example \ref{ex_cusp_not}. The dual graph associated with $\conf$ was depicted in Figure \ref{fig_cusp}. In this case, the cone of curves ${\rm NE}(\widetilde{X})$ is generated by the classes of the divisors in the following set:
    $$
    \left\{\widetilde{L},\widetilde{C}_1,\widetilde{C}_2,\widetilde{E}_{0\,1},\widetilde{E}_{0\,2},\widetilde{E}_{0\,3},\widetilde{E}_{0\,4},\widetilde{E}_{0\,5},\widetilde{E}_{0\,6},\widetilde{E}_{1\,1},\widetilde{E}_{1\,2},\widetilde{E}_{2\,1} \right\}.
    $$
    In addition, the nef cone $\operatorname{Nef}(\widetilde{X})$ is generated by the classes of the divisors $D_{0\,m(0)}$ and $D_{0\,m(0)}^{J,\bar{\ell}_J}$ introduced in Example \ref{ex_cusp_not}.
\end{example}

\section{Rational surfaces whose Cox ring is finitely generated}\label{sec:Cox_ring}

In this section we introduce a family of rational surfaces given by configurations as those introduced in \eqref{eq_conf}, and prove that the surfaces in this family have a finitely generated Cox ring.

\medskip

For a start, we recall the concepts of proximity matrix and $P$-sufficient configuration, which will be useful. Let $\mathcal{C}=\{q_i\}_{i=1}^t$ be a configuration over $\mathbb{P}^2$. The matrix $\textbf{P}_\conf=(p_{i\,j})_{1\leq i,j\leq t}$, where $p_{i\, i}=1,\, p_{i\,j}=-1$ when $q_i\to q_j,$ and $p_{i\,j}=0$ otherwise, is called  the \emph{proximity matrix} of $\conf$. Let $Z$ be the surface obtained by blowing-up at the points in $\conf$ and set $\mathbf{M}_\conf=(\textbf{P}^{t}_\conf)^{-1}=(m_{i\,j})_{1\leq i,j\leq t}$. For $i\in\{1,\ldots,t\}$, consider the divisor on $Z$, $Q_i:=\sum_{j=1}^t m_{i\,j}E_{j}^*$, and define the $t\times t$ matrix $\textbf{G}_\mathcal{C}=(g_{i\,j})$, $$g_{i\,j}=-9Q_i\cdot Q_j-(K_Z\cdot Q_i)(K_Z\cdot Q_j), \text{ }i,j\in\{ 1,\ldots,t\},$$ where $K_Z$ denotes a canonical divisor on $Z$. Note that $\textbf{G}_\mathcal{C}$ is symmetric.

\begin{definition} [{\cite[Definition 2]{GalMon2005-JA}}]
A configuration $\mathcal{C}$ as above is \emph{$P$-sufficient} if  $\textbf{x}\,G_\mathcal{C}\,\textbf{x}^t>0$ for all $\textbf{x}\in\mathbb{R}^t\setminus\{\textbf{0}\}$ with nonnegative coordinates.
\end{definition}

\begin{remark}\label{remark:P_suff}
According to \cite[Corollary 2]{GalMon2004}, if the configuration $\mathcal{C}$ is a chain, then $\mathcal{C}$ is $P$-sufficient if and only if the $(t,t)$-entry of $\textbf{G}_\conf,g_{t\,t},$ is positive. Moreover, in this last case, all the entries of $\textbf{G}_\conf$ are positive. Finally, a square matrix whose entries are nonnegative and whose diagonal entries are strictly positive satisfies the arithmetic condition involved in the definition of $P$-sufficiency.
\end{remark}

Next, we describe some consequences of $P$-sufficiency. We first recall that a divisor $D$  on a surface $Z$ is \emph{big} if the rational map $\phi_m:Z\dashrightarrow\mathbb{P}H^0(Z,\mathcal{O}_Z(mD))$ is birational onto its image for some $m>0$ (see \cite[Lemma 2.2.3]{Laz1} for an equivalent condition).

\begin{lemma}\cite[Lemma 3]{GalMon2005-JA}\label{prop_Psuf}
    Keep the notation as introduced before. Let $\conf$ be a $P$-sufficient configuration giving rise to a surface $Y$ and let $D$ be a nonzero effective divisor on $Y$ such that $D^2\geq 0$ and $D\cdot \widetilde{E}_{i}\geq 0$ for all $1\leq i\leq t$. Then $K_{Y}\cdot D<0 $.
\end{lemma}

\begin{theorem}\label{prop_Kbig}
    Keep the notation introduced before. If $\conf$ is a $P$-sufficient configuration, then  $-K_{Y}$ is a big divisor.
\end{theorem}

\begin{proof}
    By \cite[Theorem 1]{GalMon2004} the cone of curves $\operatorname{NE}(Y)$ is finitely generated and therefore the nef cone ${\rm Nef}(Y)$ is also finitely generated. For every $0\neq D\in{\rm Nef}(Y)$, we have that $D^2\geq 0$ and $D\cdot \widetilde{E}_{j\,i}\geq 0$. Note that ${\rm Nef}(Y)\subseteq\overline{{\rm NE}(Y)}={\rm NE}(Y)$. Hence, there exists a nonnegative integer $m$ such that $mD$ is effective and, by Lemma \ref{prop_Psuf}, $-K_{Y}\cdot D>0$. The facts that ${\rm Nef}(Y)^\vee={\rm NE}(Y)$ and
    $$[-K_{Y}]\in{\rm Nef}(Y)^\vee\setminus\{D|D\cdot N=0,\, N \in{\rm Nef}(Y)\}$$
    complete the proof.
\end{proof}

\begin{remark}\label{rem:Cox_ring}
    Let $Y$ be a rational surface given by a $P$-sufficient configuration, then Theorem \ref{prop_Kbig} and \cite[Theorem 1]{TesVarVel} prove that  $\operatorname{Cox}(Y)$ is finitely generated.
\end{remark}

Now consider the notation used in Section \ref{sec:nef_cone}. Set $\mathcal{C}=\{q_u\}_{u=1}^s= \bigcup_{j=0}^k\left[\{p_{j\,\ell_j}\}_{\ell_j=1}^{s_j}\right],$ $ \text{where } s=\sum_{j=0}^ks_j$, a configuration of infinitely near points as in \eqref{eq_conf}. The corres\-pondence between points $q$ and $p$ is as follows:
$$
\begin{array}{ll}
   q_{\ell_0}=p_{0\,\ell_0},  & 1\leq\ell_0\leq s_0, \\
    q_{s_0+\cdots +s_j+\ell_{j+1}}=p_{j+1\,\ell_{j+1}}, & 0\leq j\leq k-1,1\leq \ell_{j+1}\leq s_{j+1}.
\end{array}
$$
That is, the subindices of $q$ and $p$ are related by the bijection
$$\psi:(1,\ldots,s)\to (0\,1,\ldots,0\,s_0,1\,1,\ldots,1\,s_1,\ldots,k\,1,\ldots,k\,s_k).$$
Adopting our double notation, we write the before introduced matrix $\mathbf{M}_\conf=(\textbf{P}^{t}_\conf)^{-1}=\left(b_{u\,v}:=b_{\psi(u)}^{\psi(v)}\right)_{1\leq u,v\leq s}$. This allows us to define the divisors
$$\{Q_u\}_{u=1}^s=\bigcup_{j=0}^k\left[\{P_{j\,m(j)}\}_{m(j)=1}^{s_j}\right]$$ for $u=\psi^{-1}(j\,m(j))$ as:
$$
Q_{m(0)}:=Q_{\psi^{-1}(0\,m(0))}=\sum_{\ell_0=1}^{m(0)}b_{0\,m(0)}^{0\,\ell_0}E_{0\,m(0)}^*=:P_{0\,m(0)},
$$
$0\leq m(0)\leq s_0$, and, for $1\leq j\leq k$ and $1\leq m(j)\leq s_j,$
$$
Q_{\sum_{i=0}^{j-1}s_i+m(j)}=Q_{\psi^{-1}(j\,m(j))}:=\sum_{\ell_0=1}^{s_0}b_{j\,m(j)}^{0\,\ell_0}E_{0\,\ell_0}^*+\sum_{\ell_j=1}^{m(j)}b_{j\,m(j)}^{j\,\ell_j}E_{j\,\ell_j}^*=:P_{j\,m(j)}.
$$
Recalling the definition of germs $ \varphi_i$ before Lemma \ref{lem_for_Noether} and  following \cite[Corollary 3.1]{Lip1994}, it holds that $b_{j\,m(j)}^{j\,\ell_j}=\operatorname{mult}_{p_{j\,\ell_j}}(\varphi_{j\,m(j)})$, for $0\leq j\leq k$ and $1\leq \ell_j,m(j)\leq s_j.$ Consequently, the divisor $P_{j\,m(j)}$ coincides with the exceptional part of the divisor $D_{j\,m(j)}$ defined in \eqref{eq_div_D_0_m0} with the exception of its sign. We also adapt to our notation the previously introduced $s\times s$ matrix $\textbf{G}_\conf=\left(g_{u\,v}:=g_{\psi(u)}^{\psi(v)}\right)_{1\leq u,v\leq s}$. Here
\begin{equation}\label{eq:matrix_G_C}
  \begin{array}{rl}
  g_{u\,v} & := -9\, Q_{u} \cdot Q_{v} - (K_{\widetilde{X}} \cdot Q_{u})(K_{\widetilde{X}} \cdot Q_{v})\\
     & =  -9\, P_{j_1\,m(j_1)} \cdot P_{j_2\,m(j_2)} - (K_{\widetilde{X}} \cdot P_{j_1\,m(j_1)})(K_{\widetilde{X}}\cdot P_{j_2\,m(j_2)}),\\[2mm]
\end{array}
\end{equation}
with $\psi(u)=j_1\,m(j_1)$ and $\psi(v)=j_2\,m(j_2).$

\medskip

The surfaces we consider in this section are linked to the family of curves with a place at infinity we are going to define.

\begin{definition}\label{ams_type}
A projective curve $C$ on $\mathbb{P}^2$ is said to be of \emph{Abhyankar-Moh-Suzuki} type (AMS-type for short) if it has one place at infinity and it is rational and smooth in its affine part, that is, $C\setminus L$ is isomorphic to $\mathbb{C}$.
\end{definition}


Our main result in this section is the following one:

\begin{theorem}\label{thm:Cox_ring}
Let $C$ be an curve of AMS-type and degree $d$. Let $\mathcal{C}=\mathcal{B}\cup \mathcal{D}=\bigcup_{j=0}^k\left[\{p_{j\,\ell_j}\}_{\ell_j=1}^{s_j}\right]$ be a configuration of infinitely near points as introduced in \eqref{eq_conf}, where $\mathcal{B}=\{p_{0\,\ell_0}\}_{\ell_0=1}^{s_0}$ is the configuration of base points of the pencil $\mathcal{P}_C$ and $s_j\leq 2$ for all $1\leq j\leq k$. Then, the Cox ring of the surface $\widetilde{X}$ obtained by blowing-up at the points in $\mathcal{C}$ is finitely generated.
\end{theorem}

\begin{proof}
    We will show that the configuration $\mathcal{C}$ described in the statement is $P$-sufficient. Then, the result follows by Remark \ref{rem:Cox_ring}.

Let $X_n$ be the surface obtained by blowing-up at the points in $\mathcal{B}$. Set  $\widetilde{C}|_n$ the strict transform of $C$ on $X_n$. It holds that
    $$
    \widetilde{C}|_n\sim dE_{0\,0}^* - \sum_{\ell_0=1}^{s_0} \operatorname{mult}_{p_{0\,\ell_0}}(\varphi_{0\,s_0}) E_{0\,\ell_0}^*.
    $$
    In addition, since $\widetilde{C}|_n^2 = 0$, we have:
    \begin{equation}\label{eq_self_intersection_C}
        \sum_{\ell_0=1}^{s_0} \operatorname{mult}_{p_{0\,\ell_0}}(\varphi_{0\,s_0})^2 = d^2.
    \end{equation}
    Because the curve $C$ is rational and from the expression of the arithmetic genus of $\widetilde{C}|_n,\,p_a(\widetilde{C}|_n),$ we deduce that:
    $$
        0 = p_a(\widetilde{C}|_n) = 1 + \frac{1}{2} \left( \widetilde{C}|_n^2 + K_{X_n} \cdot \widetilde{C}|_n \right),
    $$
    which is equivalent to:
    \begin{equation}\label{eq_arithm_genus_C}
         \sum_{\ell_0=1}^{s_0}\operatorname{mult}_{p_{0\,\ell_0}}(\varphi_{0\,s_0}) = 3d - 2.
    \end{equation}

    Now consider the $s\times s$ matrix $\textbf{G}_{\mathcal{C}}=(g_{u\,v})_{1\leq u,v\leq s}$ defined in \eqref{eq:matrix_G_C}. According to Remark \ref{remark:P_suff}, to prove that the configuration $\mathcal{C}$ is $P$-sufficient
    it suffices to show that $g_{u\,v}\geq 0 $ and $g_{u\,u}>0$ for all $1\leq u,v\leq s$. Since $\mathcal{C}=\mathcal{B}\cup\mathcal{D},$ we partition the matrix $\mathbf{G}_\conf$ into four submatrices:
    \[
    \textbf{G}_{\conf}=\left(\begin{array}{cc}
     \textbf{G}_{\mathcal{B}}&  \textbf{G}_{\mathcal{B}}^{\mathcal{D}}\\
     (\textbf{G}_{\mathcal{B}}^{\mathcal{D}})^t & \textbf{G}_{\mathcal{D}}
    \end{array}\right),\]
    where $\textbf{G}_{\mathcal{B}}=(g_{u\,v})_{1\leq u,v\leq s_0}, \textbf{G}_{\mathcal{D}}=(g_{u\,v})_{s_0+1\leq u,v\leq s_k}$ and $\textbf{G}_{\mathcal{B}}^{\mathcal{D}}=(g_{u\,v})_{1\leq u\leq s_0}^{s_0+1\leq v\leq s_k}$.
    We begin with the submatrix $\textbf{G}_{\mathcal{B}}$ corresponding with the configuration $\mathcal{B}$. Since  $\mathcal{B}$ is a chain,  by Remark \ref{remark:P_suff} it suffices to check that the value $g_{s_0\,s_0}$ is positive. Indeed, using Equalities \eqref{eq_arithm_genus_C} and \eqref{eq_self_intersection_C},
    \begin{equation}\label{eq:g_s0_s0}
    \begin{array}{rl}
         g_{0\,s_0}^{0\,s_0} &=-9\, P_{0\,s_0} \cdot P_{0\,s_0} - (K_{\widetilde{X}} \cdot P_{0\,s_0})(K_{\widetilde{X}}\cdot P_{0\,s_0})\\
         &=9\sum_{\ell_0=1}^{s_0} \operatorname{mult}_{p_{0\,\ell_0}}(\varphi_{0\,s_0})^2-\left(\sum_{\ell_0=1}^{s_0} \operatorname{mult}_{p_{0\,\ell_0}}(\varphi_{0\,s_0})\right)^2\\[2mm]
         &=9d^2-(3d-2)^2=12d-4>0,
         \end{array}
    \end{equation}
    which yields the desired inequality.

   Next, consider the matrix $\textbf{G}_{\mathcal{D}}=(g_{u\,v})_{s_0+1\leq u,v\leq s_k}$ where $\psi(u)=j_1\,\ell_{j_1}$ and $\psi(v)=j_2\,\ell_{j_2}.$ Notice that
    by \eqref{eq:matrix_G_C} it is clear that $\textbf{G}_{\mathcal{D}}$ is symmetric. To prove that each entry of $\textbf{G}_{\mathcal{D}}$ is nonnegative we distinguish two cases: $j_1=j_2$ and $j_1\neq j_2.$ First, assume that $j_1=j=j_2$, and one can also suppose that $\ell_{j_1}=\ell_1\leq \ell_2=\ell_{j_2},$ since $\textbf{G}_{\mathcal{D}}$ is symmetric. Then,
    \begin{align*}
         g_{j\,\ell_1}^{j\,\ell_2}=&-9\, P_{j\,\ell_1} \cdot P_{j\,\ell_2} - (K_{\widetilde{X}} \cdot P_{j\,\ell_2})(K_{\widetilde{X}} \cdot P_{j\,\ell_2})\\
         =&-9\left(-\sum_{\ell_0=1}^{s_0} \operatorname{mult}_{p_{0\,\ell_0}}(\varphi_{0\,s_0})^2-\ell_1\right)\\ &\,-\left(\sum_{\ell_0=1}^{s_0}\operatorname{mult}_{p_{0\,\ell_0}}(\varphi_{0\,s_0})+\ell_1\right)\left(\sum_{\ell_0=1}^{s_0}\operatorname{mult}_{p_{0\,\ell_0}}(\varphi_{0\,s_0})+\ell_2\right)\\[2mm]
         =&\ 9(d^2+\ell_1)-(3d-2+\ell_1)(3d-2+\ell_2),
    \end{align*}
    where the first equality follows from 
    \eqref{eqn_Ds_rela} and the second one from Equalities \eqref{eq_self_intersection_C} and \eqref{eq_arithm_genus_C}.

    To consider elements in the diagonal, let $\ell_1=\ell_2=\ell$ be and then
    \begin{equation}\label{eq:entry_g_j_l}
    g_{j\,\ell}^{j\,\ell}=9(d^2+\ell)-(3d-2+\ell)(3d-2+\ell)=d(12-6\ell) -\ell^2+13\ell-4,
    \end{equation}
    which is strictly positive for any value of $d$ because $0\leq \ell\leq 2.$ The remaining considered values in the matrix are nonnegative, in fact $g_{j\,1}^{j\,2}>0$.

    To conclude our study of the entries of the matrix $\textbf{G}_{\mathcal{D}}$, we suppose $j_1\neq j_2$ and then
    \begin{align*}
        g_{j_1\,\ell_{j_1}}^{j_2\,\ell_{j_2}}=&-9\,P_{j_1\,\ell_1}\cdot P_{j_1\,\ell_1} - (K_{\widetilde{X}} \cdot P_{j_2\,\ell_{j_2}})(K_{\widetilde{X}} \cdot P_{j_2\,\ell_{j_2}})\\
        =& \ 9\sum_{\ell_0=1}^{s_0} \operatorname{mult}_{p_{0\,\ell_0}}(\varphi_{0\,s_0})^2 \\ & -\left(\sum_{\ell_0=1}^{s_0}\operatorname{mult}_{p_{0\,\ell_0}}(\varphi_{0\,s_0})+\ell_1\right)\left(\sum_{\ell_0=1}^{s_0}\operatorname{mult}_{p_{0\,\ell_0}}(\varphi_{0\,s_0})+\ell_2\right)\\[2mm]
        =& \ 9 d^2 -(3d-2+\ell_1)(3d-2+\ell_2)\\
        =& \ d(12-3(\ell_1+\ell_2)) - \ell_1\ell_2 +2 (\ell_1+\ell_2)-4,
    \end{align*}
    which is nonnegative for any $d$ because $0\leq \ell_1\leq 2$ and $0\leq \ell_2\leq 2$.

    It remains to consider the matrix $\textbf{G}_{\mathcal{B}}^{\mathcal{D}}=(g_{u\,v})_{1\leq u\leq s_0}^{s_0+1\leq v\leq s_k}$ and to show that its entries are not negative. In this case  $\psi(u)=0\,\ell_0$ and $\psi(v)=j\,\ell_j, \ 0\leq j\leq k\ \text{ and } \ 1\leq \ell_j\leq s_j\leq 2$. Firstly, we observe that one can partition $\textbf{G}_{\mathcal{B}}^{\mathcal{D}}$ as $\left(\textbf{G}_{\mathcal{B}}^{\mathcal{D}_1}|\cdots |\textbf{G}_{\mathcal{B}}^{\mathcal{D}_k}\right)$, where $\textbf{G}_{\mathcal{B}}^{\mathcal{D}_j},1\leq j\leq k,$ has entries  $g_{0\,\ell_0}^{j\,\ell_j},$ where $1\leq \ell_0\leq s_0$ and $1\leq \ell_j\leq s_j$. Now the configuration $\mathcal{C}$ can be seen as a nondisjoint union of chains $\mathcal{H}_j=\{p_{0\,\ell_0}\}_{\ell_0=1}^{s_0}\cup\{p_{j\,\ell_j}\}_{\ell_j=1}^{s_j},$ with $1\leq j\leq k$ and, by construction the submatrices $\textbf{G}_{\mathcal{B}}^{\mathcal{D}_j}$ are also submatrices of the matrix $\textbf{G}_{\mathcal{H}_j}$ corresponding with the chain $\mathcal{H}_j$. By Remark \ref{remark:P_suff}, it suffices to prove that $g_{0\, s_0}^{j\, s_j}>0$ for $1\leq j\leq k$ in order to know if all the entries of $\textbf{G}_{\mathcal{B}}^{\mathcal{D}_j}$ are positive. The positivity of these values $g_{0\, s_0}^{j\, s_j}$ follows from the positivity of those considered in Equality \eqref{eq:entry_g_j_l} when studying the matrix $\mathbf{G}_\mathcal{D}$ and since we have a chain. This completes the proof that $\mathcal{C}$ is $P$-sufficient which, by Remark \ref{rem:Cox_ring}, shows that the Cox ring $\operatorname{Cox}(\widetilde{X})$ is finitely generated.
\end{proof}

\begin{example}
The curve $C$ considered in Example \ref{ex_cusp_not} is of AMS-type. Hence, by Theorem \ref{thm:Cox_ring}, the Cox ring of the surface $\widetilde{X}$ determined by the configuration $\mathcal{C}$ in that example is finitely generated.
\end{example}

\begin{remark}\label{remark:redundant}
    Redundant blowups and its relation to the finite generation of Cox rings of rational surfaces were described in \cite{HwaPark2015,HwaPark2016}. In those papers, the authors study redundant blowups of rational surfaces with big anticanonical divisors and conclude that a redundant blowup preserves the finite generation of a Cox ring under a suitable condition. The aim of this remark is to show that configurations in this section with $g\leq 2$ do not fall within the scope of the aforementioned papers.

    It is well known that a pseudoeffective divisor $D$ on a rational projective surface $Z$ has a \emph{Zariski decomposition} $D=P_D+N_D$, where $P_D$ is a nef $\mathbb{Q}$-divisor, $N_D$ an effective $\mathbb{Q}$-divisor such that its components define a negative definite matrix and $P_D\cdot N_D=0$. Here, $P_D$ (respectively, $N_D$) is the positive (respectively, negative) part of $D$. Let $-K_Z$ be the anticanonical divisor of $Z$. Suppose that $-K_Z$ is pseudoeffective and denote by $-K_Z=P+N$ its Zariski decomposition. Let $\pi: \tilde{Z}\to Z$ be the blowup at a point $p\in Z$ and $E$ the corresponding exceptional divisor. This point is called \emph{redundant} if $\operatorname{mult}_p N\geq 1.$ In this case the blowup $\pi$ is said to be a \emph{redundant blowup} and the exceptional divisor $E$ is a \emph{redundant curve}.

    Assume that the anticanonical divisor $-K_Z$ is big. If $\pi$ is a redundant blowup, then the anticanonical divisor $-K_{\widetilde{Z}}$ is also big, although it is possible for both divisors  $-K_Z$ and $-K_{\widetilde{Z}}$ to be big without $\pi$ being redundant \cite[Remark 5.2]{HwaPark2015}. When $-K_{\widetilde{Z}}$ is a big divisor, an alternative way to see that the exceptional curve $E$ is redundant is as follows: Denoting the Zariski decomposition of $-K_{\widetilde{Z}}$ by
    $-K_{\widetilde{Z}}=\widetilde{P}+\widetilde{N} $, the blowup $\pi$ is redundant if and only if $\widetilde{P}\cdot E=0$ (see \cite[see Remark 2.3]{HwaPark2015}). 

    Returning to our situation, let $X=X_n$ be the surface obtained by blowing-up at the points of the configuration $\mathcal{B}=\{p_i\}_{i=1}^n$ of base points of a pencil $\mathcal{P}_C$ given by a curve $C$ of AMS-type. The anticanonical divisor $-K_{X}$ of $X$ satisfies $ -K_{X}\sim 3E_{0\, 0}^* -\sum_{i=1}^nE_{0\,i}$. In addition, as a consequence of Theorem \ref{thm_Campillo} (or Proposition \ref{prop_dual} and Theorem \ref{thm_conos}), the cone of curves $\operatorname{NE}(X)$ is spanned by the classes of the strict transforms of the projective line at infinity and the exceptional divisors, while the nef cone $\operatorname{Nef}(X)$ is spanned by the classes of the previously introduced divisors $D_{0\,m(0)},0\leq m(0)\leq n$.

    Thus, when $g\leq 2$ (meaning that the dual graph of $\mathcal{B}$ has 1 or 2 subgraphs), the Zariski decomposition of the anticanonical divisor $ -K_{X}=P_{-K_{X}}+N_{-K_{X}}$ is:
    $$
    P_{-K_{X}}=\left(3-\dfrac{(D_{0\,n}\cdot E_{0\,0}^*)\left(D_{0\,n}\cdot(\sum_{\ell_0=1}^nE_{0\, \ell_0}^*)\right)}{D_{0\,n}\cdot P_{0\,n}} \right)E_{0\,0}^*+\dfrac{D_{0\,n}\cdot(\sum_{\ell_0=1}^nE_{0\, \ell_0}^*)}{D_{0\,n}\cdot P_{0\,n}}D_{0\,n}
    $$
    and
    $$
    N_{-K_{X}}=\sum_{i=1}^{n-1}\left(\dfrac{(D_{0\,n}\cdot P_{0\,i})(D_{0\,n}\cdot(\sum_{\ell_0=1}^nE_{0\, \ell_0}^*))}{D_{0\,n}\cdot P_{0\,n}}-D_{0\,i}\cdot\left(\sum_{\ell_0=1}^nE_{0\, \ell_0}^*\right)\right)\widetilde{E}_{0\,i}.
    $$
    The proof is based on local properties of the AMS-type curves and it is omitted.

    The point $p_{1\,1}\in E_{0\,n}$ of any configuration $\conf$ as described in this section is not redundant. Indeed, if $p_{1\,1}$ were redundant, then $\operatorname{mult}_{p_{1\,1}}(N_{-K_{X}})\geq 1$ and then $p_{1\,1}\in E_{0\,n-1}$. However, the point $p_{1\,1}\in\conf$ is free, a contradiction.

    We conclude by presenting the Zariski decomposition of the anticanonical divisor $-K_{X_6}$ corresponding to the surface $X_6$ obtained by blowing-up at the configuration $\mathcal{B}=\{p_{0\,\ell_0}\}_{\ell_0=1}^6$ in Example \ref{ex_cusp_not}:
    $$
    -K_{X_6}=\dfrac{2}{3}E^*_{0\,0} + \dfrac{7}{9}D_{0\,n} +\dfrac{5}{9}\widetilde{E}_{0\,1}+\dfrac{1}{3}\widetilde{E}_{0\,2}+\dfrac{2}{3}\widetilde{E}_{0\,3}+\dfrac{4}{9}\widetilde{E}_{0\,4}
    +\dfrac{2}{9}\widetilde{E}_{0\,5}.
    $$
\end{remark}

\section*{Conflict of interest}
The authors declare they have no conflict of interest.

\bibliographystyle{plain}
\bibliography{MIBIBLIO}

@Article{ArtLaf,
author = {M. Artebani and A. Laface},
title = {Cox rings of surfaces and anticanonical {I}itaka dimension},
journal = {Adv. Math.},
year = {2011},
OPTkey = {•},
volume = {226},
OPTnumber = {•},
pages = {5252-5267},
OPTmonth = {•},
OPTnote = {•},
OPTannote = {•}
}

@Article{ArtPer,
author = {M. Artebani and S. P\'erez-Garbayo},
title = {Cox rings of nef anticanonical rational surfaces},
journal = {Adv. Math.},
year = {2025},
OPTkey = {•},
volume = {475},
OPTnumber = {•},
pages = {110328},
OPTmonth = {•},
OPTnote = {•},
OPTannote = {•}
}

@book{ArzDerHauLaf,
author={I. Arzhantsev and U. Derenthal and J. Hausen and A. Laface},
title={Cox Rings},
place={Cambridge},
series={Cambridge Studies in Advanced Mathematics},
publisher={Cambridge University Press},
year={2014},
collection={Cambridge Studies in Advanced Mathematics}}

@article {BirCasHacMcKer,
AUTHOR = {C. Birkar and P. Cascini and C.D. Hacon and J. McKernan},
TITLE = {Existence of minimal models for varieties of log general type},
JOURNAL = {J. Amer. Math. Soc.},
FJOURNAL = {},
VOLUME = {23},
YEAR = {2010},
NUMBER = {2},
PAGES = {405--468},
ISSN = {},
MRCLASS = {},
MRNUMBER = {},
MRREVIEWER = {},
DOI = {},
URL = {},
}

@article {B,
AUTHOR = {R.I. Bogdanov},
TITLE = {Bifurcation of the limit cycle of a family of plane vector fields},
JOURNAL = {Selecta Math. Soviet},
VOLUME = {1},
YEAR = {1981},
NUMBER = {},
PAGES = {373--387},
}

@article{Bre,
  title={Semigroups corresponding to algebroid branches in the plane},
  author={Bresinsky, H.},
  journal={Proc. Amer. Math. Soc.},
  pages={381--384},
  year={1972},
}

@incollection {CamGonz,
    AUTHOR = {A. Campillo and G. Gonz\'{a}lez-Sprinberg},
TITLE = {On characteristic cones, clusters and chains of infinitely near points},
BOOKTITLE = {Singularities ({O}berwolfach, 1996)},
SERIES = {Progr. Math.},
VOLUME = {162},
PAGES = {251--261},
PUBLISHER = {Birkh\"{a}user, Basel},
YEAR = {1998},
}

@Article{CamPilReg-2002,
author = {A. Campillo and O. Piltant and A.J. Reguera},
title = {Cones of curves and of line-bundles on surfaces associated with curves having one place at infinity},
journal = {Proc. London Math. Soc.},
year = {2002},
volume = {84},
pages = {559--580}
}

@Book{Cas,
author = {E. Casas-Alvero},
title = {{S}ingularities of {P}lane {C}urves},
publisher = {Cambridge Univ. Press},
year = {2000},
OPTkey = {•},
volume = {276},
OPTnumber = {•},
series = {London Math. Soc. Lecture Note Ser.},
OPTaddress = {•},
OPTedition = {•},
OPTmonth = {•},
OPTnote = {•},
OPTannote = {•}
}

@article {Casc2021,
    AUTHOR = {C. Cascini},
     TITLE = {New directions in the {M}inimal {M}odel {P}rogram},
   JOURNAL = {Boll. Unione Mat. Ital. },
  FJOURNAL = {Bollettino dell'Unione Matematica Italiana},
    VOLUME = {14},
      YEAR = {2021},
     PAGES = {179--190},
      }

@Article{CL,
author = {C.~Christopher and J.~Llibre},
title = {Algebraic aspects of integrability for polynomial systems},
journal = {Qual. Th. Planar Differ. Equ.},
year = {1999},
OPTkey = {•},
volume = {1},
OPTnumber = {•},
pages = {71--95},
OPTmonth = {•},
OPTnote = {•},
OPTannote = {•}
}

@Article{Cox,
author = {D. Cox},
title = {The homogeneous coordinate ring of a toric variety},
journal = {J. Algebraic Geom.},
year = {1995},
volume = {4},
pages = {17--50}
}

@Article{RosaFrisMus2017,
author = {B.L. de la Rosa-Navarro and J.B. Fr\'{\i}as-Medina and  M. Lahyane},
title = {Rational surfaces with finitely generated {C}ox rings and very high {P}icard numbers},
journal = {{RACSAM}},
year = {2017},
volume = {111},
OPTnumber = {•},
pages = {297--306},
OPTmonth = {•},
OPTnote = {•},
OPTannote = {•}
}

@article {RosaFrisMus2023,
    AUTHOR = {B.L. de la Rosa-Navarro and J.B. Fr\'{i}as-Medina
              and M. Lahyane},
     TITLE = {On the effective, nef, and semi-ample monoids of blowups of {H}irzebruch surfaces at collinear points},
   JOURNAL = {Canad. Math. Bull.},
  FJOURNAL = {Canadian Mathematical Bulletin},
  publisher={Canadian Mathematical Society},
    VOLUME = {66},
      YEAR = {2023},
    NUMBER = {4},
     PAGES = {1179--1193},
      OPTISSN = {},
   OPTMRCLASS = {},
  OPTMRNUMBER = {},
OPTMRREVIEWER = {},
       OPTDOI = {},
       OPTURL = {},}

@Article{D,
author = {F. Dumortier},
title = {Singularities of vector fields on the plane},
journal = {J. Differential Equations},
year = {1977},
volume = {23},
pages = {53--106},
OPTmonth = {•},
OPTnote = {•},
OPTannote = {•}
}

@Article{F,
author = {R. A. Fisher},
title = {The wave of advance of advantageous genes},
journal = {Ann. Eugenics},
year = {1937},
OPTkey = {•},
volume = {7},
OPTnumber = {•},
pages = {355--369},
OPTmonth = {•},
OPTnote = {•},
OPTannote = {•}
}

@Article{FriLahy2021,
author = {J.B. Fr\'{\i}as-Medina and M. Lahyane},
title = {The effective monoids of the blow-ups of {H}irzebruch sufaces at points in general position},
journal = {Rend. Circ. Mat. Palermo, (2). Ser.},
year = {2021},
OPTkey = {•},
volume = {70},
OPTnumber = {},
pages = {167--197},
OPTmonth = {•},
OPTnote = {•},
OPTannote = {•},
}

@Article{FulMur,
author = {M. Fulger and T. Murayama},
title = {New constructions of nef classes on self-products
of curves},
journal = {Math. Z.},
year = {2022},
OPTkey = {•},
volume = {302},
OPTnumber = {•},
pages = {1239--1265},
OPTmonth = {•},
OPTnote = {•},
OPTannote = {•}
}

@incollection {Ful2004,
AUTHOR = {W. Fulton},
TITLE = {Adjoints and {M}ax {N}oether's {F}undamentalsatz},
BOOKTITLE = {Algebra, arithmetic and geometry with applications ({W}est {L}afayette, {IN}, 2000)},
PAGES = {301--313},
PUBLISHER = {Springer, Berlin},
YEAR = {2004},
}

@book {Ful2,
AUTHOR = {W. Fulton},
TITLE = {Introduction to {T}oric {V}arieties},
SERIES = {Ann. of Math. Stud.},
VOLUME = {131},
OPTNOTE = {The William H. Roever Lectures in Geometry},
PUBLISHER = {Princeton Univ. Press, Princeton, NJ},
YEAR = {1993},
OPTPAGES = {xii+157},
OPTISBN = {0-691-00049-2},
OPTMRCLASS = {14M25 (14-02 14J30)},
OPTMRNUMBER = {1234037},
OPTMRREVIEWER = {T. Oda},
OPTDOI = {10.1515/9781400882526},
OPTURL = {https://doi.org/10.1515/9781400882526},
}

@article {GalMon2004,
AUTHOR = {C. Galindo and F. Monserrat},
TITLE = {On the cone of curves and of line bundles of a rational
              surface},
JOURNAL = {Int. J. Math.},
FJOURNAL = {International Journal of Mathematics},
VOLUME = {15},
YEAR = {2004},
NUMBER = {4},
PAGES = {393--407},
}

@Article{GalMonMor2025-RAC,
author = {C. Galindo and F. Monserrat and C.J. Moreno-\'Avila},
title = {The cone of curves and the {C}ox ring of rational
surfaces over {H}irzebruch surfaces.},
journal = {Rev. Real Acad. Cienc. Exactas
Fis. Nat. Ser. A-Mat.},
year = {2025},
OPTkey = {•},
volume = {119},
OPTnumber = {•},
pages = {92},
OPTmonth = {•},
OPTnote = {•},
OPTannote = {•}
}

@article {GalMon2005,
AUTHOR = {C. Galindo and F. Monserrat},
TITLE = {The cone of curves associated to a plane configuration},
JOURNAL = {Comment. Math. Helv.},
FJOURNAL = {Commentarii Mathematici Helvetici. A Journal of the Swiss
              Mathematical Society},
VOLUME = {80},
YEAR = {2005},
NUMBER = {1},
PAGES = {75--93},
}

@Article{GalMon2005-JA,
author = {C. Galindo and F. Monserrat},
title = {The total coordinate ring of a smooth projective surface},
journal = {J. Algebra},
year = {2005},
OPTkey = {•},
volume = {284},
OPTnumber = {•},
pages = {91--101},
OPTmonth = {•},
OPTnote = {•},
OPTannote = {•}
}

@Article{GalMon2016,
AUTHOR = {C. Galindo and F. Monserrat},
TITLE = {The cone of curves and the {C}ox ring of rational surfaces given by divisorial valuations},
JOURNAL = {Adv. Math.},
FJOURNAL = {•},
VOLUME = {290},
YEAR = {2016},
NUMBER = {2016},
PAGES = {1040--1061},
ISSN = {},
MRCLASS = {•},
MRNUMBER = {•},
MRREVIEWER = {•},
DOI = {},
URL = {},
}

@Article{G,
AUTHOR = {I.A. Garc\'\i a},
TITLE = {Transcendental Limit Cycles Via the Structure of Arbitrary Degree Invariant Algebraic Curves of Polynomial Planar Vector Fields},
JOURNAL = {Rocky Mountain J. Math.},
FJOURNAL = {•},
VOLUME = {35},
YEAR = {2005},
NUMBER = {•},
PAGES = {501-515},
ISSN = {•},
MRCLASS = {•},
MRNUMBER = {•},
MRREVIEWER = {•},
DOI = {•},
URL = {https://arxiv.org/abs/2105.01483},
}

@article {HacMck,
    AUTHOR = {Hacon, C. D. and McKernan, J.},
     TITLE = {Existence of minimal models for varieties of log general type.
              {II}},
   JOURNAL = {J. Amer. Math. Soc.},
  FJOURNAL = {Journal of the American Mathematical Society},
    VOLUME = {23},
      YEAR = {2010},
    NUMBER = {2},
     PAGES = {469--490},
}

@article {HarbLahy,
    AUTHOR = {Harbourne, B. and Lahyane, M.},
     TITLE = {Irreducibility of {$-1$}-classes on anticanonical rational
              surfaces and finite generation of the effective monoid},
   JOURNAL = {Pacific J. Math.},
  FJOURNAL = {Pacific Journal of Mathematics},
    VOLUME = {218},
      YEAR = {2005},
    NUMBER = {1},
     PAGES = {101--114},
}

@article {HauKeiLaf,
    AUTHOR = {J. Hausen and S. Keicher and A. Laface},
     TITLE = {Computing {C}ox rings},
   JOURNAL = {Math. Comp.},
  FJOURNAL = {Mathematics of Computation},
    VOLUME = {85},
      YEAR = {2016},
    NUMBER = {297},
     PAGES = {467--502},
      }

@Article{HuKe,
author = {Y. Hu and S. Keel},
title = {Mori dream spaces and {GIT}},
journal = {Michigan Math. J.},
year = {2000},
OPTkey = {•},
volume = {48},
OPTnumber = {•},
pages = {331--348},
OPTmonth = {•},
OPTnote = {•},
OPTannote = {•}
}

@article {HwaPark2015,
    AUTHOR = {Hwang, D. and Park, J.},
     TITLE = {Redundant blow-ups of rational surfaces with big anticanonical
              divisor},
   JOURNAL = {J. Pure Appl. Algebra},
  FJOURNAL = {Journal of Pure and Applied Algebra},
    VOLUME = {219},
      YEAR = {2015},
    NUMBER = {12},
     PAGES = {5314--5329},
     }

@article {HwaPark2016,
    AUTHOR = {Hwang, D. and Park, J.},
     TITLE = {Cox rings of rational surfaces and redundant blow-ups},
   JOURNAL = {Trans. Amer. Math. Soc.},
  FJOURNAL = {Transactions of the American Mathematical Society},
    VOLUME = {368},
      YEAR = {2016},
    NUMBER = {11},
     PAGES = {7727--7743},
      }

@article {Kaw,
AUTHOR = {Y. Kawamata},
TITLE = {The cone of curves of algebraic varieties},
JOURNAL = {Ann. Math. (2)},
FJOURNAL = {Annals of Mathematics. Second Series},
VOLUME = {119},
YEAR = {1984},
NUMBER = {3},
PAGES = {603--633},
OPTISSN = {0003-486X},
OPTMRCLASS = {14E05 (14E30 14J30)},
OPTMRNUMBER = {744865},
OPTMRREVIEWER = {Ulf Persson},
OPTDOI = {10.2307/2007087},
OPTURL = {https://doi.org/10.2307/2007087},
}

@Book{KolMor,
author = {J. Kollar and S. Mori},
title = {{B}irational {G}eometry of {A}lgebraic {V}arieties},
publisher = {Cambridge University Press, Cambridge},
year = {1998},
OPTkey = {•},
volume = {134},
OPTnumber = {•},
series = {Cambridge Tracts in Math.},
OPTaddress = {•},
OPTedition = {•},
OPTmonth = {•},
OPTnote = {•},
OPTannote = {•}
}

@Article{LafUga,
author = {A. Laface and L. Ugaglia},
title = {On Blowing Up Minimal
Toric Surfaces},
journal = {International Mathematics
Research Notices},
year = {2025},
OPTkey = {•},
volume = {2025},
number = {6},
pages = {rnaf052},
OPTmonth = {•},
OPTnote = {•},
OPTannote = {•}
}

@Article{LafVel,
author = {A. Laface and M. Velasco},
title = {A survey on {C}ox rings},
journal = {Geom. Dedicata},
year = {2009},
OPTkey = {•},
volume = {139},
OPTnumber = {•},
pages = {269--287},
OPTmonth = {•},
OPTnote = {•},
OPTannote = {•}
}

@Book{Laz1,
author = {R. Lazarsfeld},
title = {Positivity in {A}lgebraic {G}eometry {I}. {C}lassical {S}etting: {L}ine {B}undles and {L}inear {S}eries},
publisher = {Springer-Verlag, Berlin},
year = {2004},
OPTkey = {•},
volume = {48},
OPTnumber = {•},
OPTseries = {Ergebnisse der Mathematik und ihrer Grenzgebiete. 3. Folge. A Series of Modern Surveys in Mathematics [Results in Mathematics and Related Areas. 3rd Series. A Series of Modern Surveys in Mathematics]},
OPTaddress = {•},
OPTedition = {•},
OPTmonth = {•},
OPTnote = {},
OPTannote = {•}
}

@incollection {Lip1994,
    AUTHOR = {Lipman, Joseph},
     TITLE = {Proximity inequalities for complete ideals in two-dimensional
              regular local rings},
 BOOKTITLE = {Commutative algebra: syzygies, multiplicities, and birational
              algebra ({S}outh {H}adley, {MA}, 1992)},
    SERIES = {Contemp. Math.},
    VOLUME = {159},
     PAGES = {293--306},
 PUBLISHER = {Amer. Math. Soc., Providence, RI},
      YEAR = {1994},
      ISBN = {0-8218-5188-8},
   MRCLASS = {13H05 (13F30 13H15)},
  MRNUMBER = {1266187},
MRREVIEWER = {Bernard\ L.\ Johnston},
       DOI = {10.1090/conm/159/01512},
       URL = {https://doi.org/10.1090/conm/159/01512},
}

@article {Moh,
    AUTHOR = {Moh, T. T.},
     TITLE = {On analytic irreducibility at {$\infty $} of a pencil of
              curves},
   JOURNAL = {Proc. Amer. Math. Soc.},
  FJOURNAL = {Proceedings of the American Mathematical Society},
    VOLUME = {44},
      YEAR = {1974},
     PAGES = {22--24},
  }

@article{Mor,
author = {S. Mori},
title={Threefolds whose canonical bundles are not numerically effective},
Journal={Ann. Math.},
Volume= {116},
Number={},
Pages={133--176},
Year={1982}
}

@Article{Nag,
author = {M. Nagata},
title = {On the 14-th problem of {H}ilbert},
journal = {Amer. J. Math.},
year = {1959},
volume = {81},
pages = {766--772},
OPTmonth = {•},
OPTnote = {•},
OPTannote = {•}
}

@Article{O,
author = {K. Odani},
title = {The limit cycle of the van der Pol Equation is Not Algebraics},
journal = {J. Differential Equations},
year = {1995},
OPTkey = {•},
volume = {115},
OPTnumber = {•},
pages = {146--152},
OPTmonth = {•},
OPTnote = {•},
OPTannote = {•}
}

@article{Ott,
author={J. Ottem},
title={On the Cox ring of $\mathbf{P}^2$
blown up in points on a line},
volume={109},
journal={Math. Scan},
year={2011},
pages={22--30}
}

@Article{T,
author = {F. Takens},
title = {Forced oscillations and bifurcations, Applications of global analysis I},
journal = {Common Math. Inst. Rijksuniversitat Utrecht},
year = {1974},
OPTkey = {•},
volume = {3},
OPTnumber = {•},
pages = {1--59},
OPTmonth = {•},
OPTnote = {•},
OPTannote = {•}
}

@Article{TesVarVel,
author = {D. Testa and A. V\'arilly-Alvarado and M. Velasco},
title = {Big rational surfaces},
journal = {Math. Ann.},
year = {2011},
OPTkey = {•},
volume = {351},
OPTnumber = {•},
pages = {95--107},
OPTmonth = {•},
OPTnote = {•},
OPTannote = {•}
}

@Article{V,
author = {C. Valls},
title = {Invariant algebraic surfaces for generalized Raychaudhuri equations},
journal = {Commun. Math. Phys.},
year = {2011},
volume = {308},
pages = {133--146}
}

@book{Zar06,
  title={The moduli problem for plane branches},
  author={O. Zariski},
  year={2006},
  publisher={University Lecture Series. American Mathematical Soc.},
  volume = {39},
}
\end{document}